\documentclass{amsart}
\usepackage{rotating}
\usepackage{enumerate}
\usepackage{amsfonts}
\usepackage{amsmath}
\usepackage{amssymb}
\usepackage{textcomp}
\usepackage{float,lscape}

\usepackage{enumitem}
\usepackage[arrow, matrix, curve, graph]{xy}
\usepackage{tikz}
\usetikzlibrary{arrows}
\usepackage[hyphens]{url}  
\newtheorem{theorem}{Theorem}[section]

\newtheorem{lemma}[theorem]{Lemma}

\newtheorem{proposition}[theorem]{Proposition}

\newtheorem{corollary}[theorem]{Corollary}

\theoremstyle{definition}

\newtheorem{definition}[theorem]{Definition}

\newtheorem{problem}[theorem]{Problem}

\newtheorem{question}[theorem]{Question}

\newtheorem{remark}[theorem]{Remark}

\newtheorem *{Theorem 1}{Theorem 1}
\newtheorem *{Theorem 2}{Theorem 2}
\newtheorem *{Theorem 5}{Theorem 5}
\newtheorem *{Theorem 3}{Theorem 3}
\newtheorem *{Theorem 4}{Theorem 4}

\newtheorem *{Problem1}{The group ring isomorphism problem (GRIP)}
\newtheorem *{Problem2}{The integral group ring isomorphism problem (IGRIP)}
\newtheorem *{Problem3}{The twisted group ring isomorphism problem (TGRIP)}
\newtheorem *{Brauer}{The Trick of Brauer}

\newcommand{\Q}{{\mathbb Q}}
\newcommand{\C}{{\mathbb C}}

\newcommand{\Gal}{{\operatorname{Gal}}}
\newcommand{\RNum}[1]{\uppercase\expandafter{\romannumeral #1\relax}}
\newcommand{\rNum}[1]{\lowercase\expandafter{\romannumeral #1\relax}}
\newcommand{\Hom}{{\operatorname{Hom}}}
\newcommand{\Aut}{{\operatorname{Aut}}}
\newcommand{\Ext}{{\operatorname{Ext}}}
\newcommand{\Irr}{{\operatorname{Irr}}}

\newcommand{\Nr}{{\operatorname{Nr}}}

\begin{document}
	\title[Twisted group ring isomorphism problem]{Twisted group ring isomorphism problem and infinite cohomology groups}
	\address{Departament of Mathematics, Free University of Brussels, 1050 Brussels, Belgium}
	\author{Leo Margolis}
	\email{leo.margolis@vub.be}

	\author{Ofir Schnabel}
	\address{Department of Mathematics, ORT Braude College, 2161002 Karmiel, Israel}
	\email{ofirsch@braude.ac.il}
\thanks{This research has been supported by the Research Foundation Flanders (FWO - Vlaanderen).}
	
	\begin{abstract}
We continue our investigation of a variation of the group ring isomorphism problem for twisted group algebras. Contrary to previous work, we include cohomology classes which do not contain any cocycle of finite order. This allows us to study the problem in particular over any field of characteristic $0$. 
We prove that there are finite groups $G$ and $H$ which can not be distinguished by their rational twisted group algebras, while $G$ and $H$ can be identified by their semi-simple twisted group algebras over other fields. This is in contrast with the fact that the structural information on $G$ which can obtained from all the semi-simple group algebras of $G$ is already encoded in its rational group algebra. 

We further show that for an odd prime $p$ there are groups of order $p^4$ which can not be distinguished by their twisted group algebras over $F$ for any field $F$ of characteristic different from $p$. On the other hand we prove that the groups constructed by E. Dade, which have isomorphic group algebras over any field, can be distinguished by their rational twisted group algebras.
We also answer a question about sufficient conditions for the twisted group ring isomorphism problem to hold over the complex numbers.
	\end{abstract}
	
	\maketitle
	
	\section{Introduction}\label{Intro}\pagenumbering{arabic} \setcounter{page}{1}
	In \cite{MargolisSchnabel} we proposed a version of the celebrated group ring isomorphism problem for twisted group rings, namely ``the
	twisted group ring isomorphism problem''.
	
	Recall that for a finite group $G$ and a commutative ring $R$, the
	group ring isomorphism problem (GRIP) asks whether the ring structure of $RG$
	determines $G$ up to isomorphism. 
	Roughly speaking the twisted group 	ring isomorphism problem (TGRIP) asks if the ring structure of all the twisted group rings of $G$
	over $R$ determines the group $G$.
  In a sense the (TGRIP) can also be understood as a question on how strongly the projective representation theory of a group over a given ring $R$ influences its structure, as was already shown in the ground laying work of Schur \cite{Schur}.
	
	We denote by $R^*$ the unit group in a ring $R$. For a $2$-cocycle $\alpha \in Z^2(G, R^*)$ the twisted group ring $R^\alpha G$ of $G$ over $R$
	with respect to $\alpha$ is the free $R$-module with basis $\{u_g\}_{g \in G}$ where the multiplication on the basis is defined via
	\[u_g u_h = \alpha(g, h) u_{gh} \ \ \text{for all} \ \ g,h \in G \]
	and any $u_g$ commutes with the elements of $R$. Owning to this definition $R^\alpha G$ is automatically associative. The ring structure of
	$R^{\alpha}G$ depends only on the cohomology class of $\alpha$ and not on the particular $2$-cocycle.
	Notice that the ring $R$ is central in the twisted group ring
	$R^{\alpha}G$ and correspondingly the associated second cohomology
	group is with respect to a trivial action of $G$ on $R^*$. See
	\cite[Chapter 3]{KarpilovskyProjective} for details. We denote the second cohomology group of $G$ with values in $R^*$ by $H^2(G,R^*)$, the cohomology class of $\alpha \in Z^2(G,F^*)$ by $[\alpha]$ and call $G$ a group base of $R^\alpha G$.
	
	\begin{definition}
		Let $R$ be a commutative ring and let $G$ and $H$ be finite groups. We say that $G \sim_R H$ if there exists a group isomorphism
		$$\psi :H^2(G,R^*)\rightarrow H^2(H,R^*)$$
		such that for any $[\alpha] \in H^2(G,R^*)$,
		$$R^{\alpha}G\cong R^{\psi (\alpha)}H.$$
	\end{definition}
	
The main problem we are interested in is the following.
	\begin{Problem3}
		For a given commutative ring $R$, determine the $\sim_R$-classes.
		Answer in particular, for which groups $G \sim_R H$ implies $G \cong H$.
	\end{Problem3}
	In \cite{MargolisSchnabel} we investigated (TGRIP) over the
	complex numbers and gave some results for families of groups, e.g.
	abelian groups, $p$-groups, groups of central type and groups of
	cardinality $p^4$ and $p^2q^2$ for $p,q$ primes. In \cite{MargolisSchnabel2} we continued our investigation of the (TGRIP) considering also other fields, but under the condition that any cohomology class contains a cocycle of finite order. This covers, in particular, twisted group algebras over the complex and real numbers and any finite field. In this situation in \cite{MargolisSchnabel2} we introduced for a group $G$ a so-called Yamazaki cover, generalizing the Schur cover of $G$ in the sense that any projective representation of $G$ is projectively equivalent to a linear representation of its Yamazaki cover. Thus a Yamazaki cover can be used to translate the problem of describing the Wedderburn decomposition of a twisted group algebra of $G$ to a problem of describing the Wedderburn decomposition of a group algebra of the Yamazaki cover of $G$. 
	
Such a translation can not be expected  when we consider cohomology classes which do not contain a cocycle of finite order, as it is the case for instance for an abelian group over the rationals. In this case semi-simple twisted group algebras might contain Wedderburn components which are not a Wedderburn component of any finite-dimensional group algebra over the given field and furthermore the cohomology groups which one has to investigate might be infinitely-generated. In this paper we develop tools to study the (TGRIP) over any field and apply those tools to several examples. 

A main motivation, based on the results in \cite{MargolisSchnabel, MargolisSchnabel2} and continuing our previous work, is also to explore:
	\begin{enumerate}
		\item[(a)] The differences between the (TGRIP) and the (GRIP).
		\item[(b)] The differences between the (TGRIP) over fields where any cohomology class contains a cocyle of finite order and the (TGRIP) over general fields.
	\end{enumerate}

The complex and rational numbers play a special role in the semi-simple linear representation theory of finite groups, in the sense that if $F$ is a field of characteristic not dividing the order of a finite group $G$, then
\begin{equation}\label{eq:CandQ}
FG \cong FH \Rightarrow \C G \cong \C H \ \ \text{and} \ \ \Q G \cong \Q H \Rightarrow FG \cong FH. 
\end{equation}
Cf. \cite[p. 220]{MargolisSchnabel2} and \cite[Theorem 14.1.9]{Pas77} for proofs of these facts. Hence one could say that in the semi-simple case of the (GRIP) the complex numbers ``know the least'', while the rationals ``know the most''. In \cite[Examples 3.1]{MargolisSchnabel2} we already observed that $\C$ is not the least knowing field for the (TGRIP) by exhibiting two non-isomorphic abelian groups $G$ and $H$ such that $G \sim_F H$ for a certain finite field $F$, while $G \sim_\C H$ and $G$ abelian implies $G \cong H$ \cite[Lemma 1.2]{MargolisSchnabel}. The second part of \eqref{eq:CandQ} leads to:

\begin{question}\label{que:QvsRest}
		Are there non-isomorphic finite groups $G$ and $H$ and a field $F$ with the characteristic of $F$ not dividing the order of $G$ such that $G \sim_{\mathbb{Q}} H$ but $G\not \sim _{F} H$?
\end{question}

Besides the absence of a Yamazaki cover, a major difficulty is that unlike in the work over $\C$ in \cite{MargolisSchnabel}, or over finite fields in \cite{MargolisSchnabel2}, over $\mathbb{Q}$ or a general field $F$, division algebras can appear in the Wedderburn decomposition of twisted group algebras, so here the Brauer group of the group base over $F$ comes into play. This leads to deep number-theoretical questions which are in general hard to answer.

As for groups of order at most $16$ the rational group algebra determines the isomorphism type of the group base, in the structural sense the smallest non-isomorphic groups which have isomorphic rational group algebras are the two non-abelian groups of order $p^3$ for $p$ an odd prime. By ~\eqref{eq:CandQ} these groups have isomorphic group algebras over any field of characteristic different from $p$. It turns out that the answer of the (TGRIP) in this case depends on a number theoretical condition.

\begin{Theorem 1}\emph{
Let $p$ be an odd prime, $G$ and $H$ the two non-isomorphic non-abelian groups of order $p^3$ and $F$ a field of characteristic different from $p$. Denote by $\zeta$ a primitive $p$-th root of unity in an extension of $F$. }

\emph{Then $G \sim_F H$ if and only if $F$ does not contain a primitive $p$-th root of unity and $\zeta$ is in the image of the norm map of the field extension $F(\sqrt[p]{\lambda}, \zeta)/F(\zeta)$ for all $\lambda \in F^*$. In particular, $G \not\sim_\Q H$.}
\end{Theorem 1} 
	
The next natural candidates to answer Question~\ref{que:QvsRest} are given by a triple of groups of order $81$ and two pairs of groups of order $p^4$ for $p$ a prime bigger than $3$. These groups have isomorphic group algebras over the rationals and we can completely solve the (TGRIP) for them.

\begin{Theorem 2} \emph{Let $p$ be an odd prime and $F$ a field of characteristic different from $p$. Denote by $G_i$ the $i$-th group of order $p^4$ in the SmallGroupLibrary of GAP.
\begin{itemize}
\item[(a)] Let $p=3$. Then
\begin{itemize}
\item[(i)] $G_8 \sim_F G_{10}$.
\item[(ii)]  $G_8 \sim_F G_9$ if and only if $F$ does not contain a primitive $3$-rd root of unity.
\end{itemize}
\item[(b)] Let $p > 3$. Then
\begin{itemize}
\item[(i)]  $G_9 \sim_F G_{10}$.
\item[(ii)] $G_7 \sim_F G_8$ if and only if $F$ does not contain a primitive $p$-th root of unity.
\end{itemize} 
\end{itemize}
In particular, Question~\ref{que:QvsRest} has a positive answer.}
\end{Theorem 2}

Theorem 2 exhibits for each odd prime $p$ a pair of groups $G$ and $H$ of order $p^4$ such that $G\sim_\Q H$, but not $G\sim_F H$ for any field $F$ of characteristic different from $p$. Hence we obtain an answer to Question~\ref{que:QvsRest}. Theorem 2 also exhibits a pair of groups of order $p^4$ which satisfy the relation of the (TGRIP) over any field of characteristic different from $p$. Note that if $F$ is a field of characteristic $p$, then $FG \not\cong FH$ for $G$ a group of order $p^4$ and $H$ a group not isomorphic to $G$ \cite[pp. 671-675]{Pas77}, hence in this case $G \not\sim_F H$.
	
Concerning the (TGRIP) over all fields simultaneously, including those of characteristic dividing the order of the group base, we study a twisted version of a question posed by Brauer \cite[Problem $2^{*}$]{Brauer63}: For finite groups $G$ and $H$ such that $FG \cong FH$ for any field $F$, is it necessary true that $G$ and $H$ are isomorphic? In \cite[Problem 1.2]{MargolisSchnabel2} we suggested a twisted version of this question.

\begin{problem}\label{TwistedBrauer}
Are there non-isomorphic finite groups $G$ and $H$ such that $G\sim_F H$ over any field $F$?
\end{problem} 

Dade \cite{Dade} constructed a family of pairs of groups $G$ and $H$ of order $p^3q^6$ which provided a counterexample for the problem of Brauer. Here $p$ and $q$ are any pair of primes satisfying $q \equiv 1 \bmod p^2$.
In \cite[Theorem 1]{MargolisSchnabel2} we proved that if $G$ and $H$ are the groups from Dade's example and they have even order, then
there exists an infinite number of fields $F$ such that $G\not \sim_{F} H$. Here we can remove the condition on the order and show the following:

\begin{Theorem 3}\emph{
For each pair of groups $G$ and $H$ from Dade's example $G\not \sim_{\Q} H$.}
\end{Theorem 3}

Therefore, we do not know the answer to Problem~\ref{TwistedBrauer}, but we know that the classical counterexample for Brauer's problem is not a counterexample to the twisted version. See Remark~\ref{rem:RoggenkampAndHertweck} for other candidates for an answer to Problem~\ref{TwistedBrauer}.

While for a finite group $G$ and a normal subgroup $N$ there exists a natural surjective ring homomorphism $FG \rightarrow F(G/N)$, this is not always the case for twisted group algebras. A crucial role in our proofs of Theorems 1-3 is played by a theorem, which is also of independent interest, and which allows to write a twisted groups algebra $F^\alpha G$ as a direct sum of twisted group algebras of $G/N$ when $N$ is a central subgroup, $\alpha$ a cocycle inflated from a cocycle of $G/N$ and the characteristic of $F$ does not divide the order of $N$. See Theorem~\ref{th:NEWINFSTUFF} for more details.

In view of a possible solution of Problem~\ref{TwistedBrauer} it would be of interest to determine if an implication like the second part of \eqref{eq:CandQ} could be formulated for given groups $G$ and $H$ and a finite number of fields, i.e.:

\begin{problem}
Does there exists a general procedure to determine some fields $F_1,\ldots,F_n$, depending on given finite groups $G$ and $H$, such that $G\sim_{F_i}H $ for $1 \leq i \leq n$ implies $G \sim_F H$ for any field $F$?
\end{problem}

We do not know an answer to this problem.

We also answer one more question for complex twisted group algebras which remained open in \cite{MargolisSchnabel}. 
It is clear that if $G\sim _F H$, for some field $F$ and finite groups $G$ and $H$, then $H^2(G,F^*) \cong H^2(H,F^*)$ and the following condition holds:\\

\textit{$F$-bijectivity condition:} There exists a set bijection $\varphi:H^2(G,F^*) \rightarrow H^2(H,F^*)$ such that
		$F^\alpha G\cong F^{\varphi(\alpha)}H$ for any $[\alpha] \in H^2(G,F^*)$.\\
		
The $F$-bijectivity condition in particular implies, that $FG \cong FH$. We study this condition especially for $F$ being the complex numbers and write $H^2(G,\C^*) = M(G)$.
	Neither the condition $M(G) \cong M(H)$ nor the $\C$-bijectivity condition by itself is sufficient to imply $G\sim_\C H$ \cite[Theorem 1.6]{MargolisSchnabel}. 	Since introducing the (TGRIP) we tried to answer the following questions. 

	\begin{question}\label{Q:suff_conditions}
		Let $G$ and $H$ be finite groups. 
		\begin{itemize}
			\item[(i)] Do $M(G) \cong M(H)$ and the $\C$-bijectivity condition together imply $G \sim_\C H$?
			\item[(ii)] If the answer to the above is negative what additional conditions will imply $G \sim_\C H$?
		\end{itemize}
	\end{question}

We answer this question in the following way.

\begin{Theorem 4}\emph{
There exist finite groups $G$ and $H$ of order 288 such that $M(G) \cong M(H)$ and the $\C$-bijectivity condition holds for $G$ and $H$, but $G \not \sim_\C H$.}

\emph{However, if $M(G) \cong M(H)$ and the $\C$-bijectivity condition holds for $G$ and $H$, and additionally $M(G)$ is cyclic or isomorphic to an elementary abelian group of order $4$ or $9$, then $G \sim_\C H$. }
\end{Theorem 4}
	
A theorem, which is important on its own, which is crucial for the proof of Theorem 4 is Theorem~\ref{th:isomiffcoprime} stating that 
\[\C ^{\alpha}G\cong \C ^{\alpha ^r}G \Leftrightarrow r \ \text{is coprime with the order of} \ [\alpha].\] 	

The paper is organized as follows.
In \S\ref{prelim} we recall definitions and results we will use throughout the paper and prove some auxiliary results which will be used later. In \S\ref{Section3} for a group $G$ and $r\in \mathbb{N}$ we prove that 
, $\C ^{\alpha}G\cong \C ^{\alpha ^r}G$ if and only if $r$ is prime to the order of $[\alpha]$ in $M(G)$. In \S\ref{sec:Proof of Theorem 4}  we prove Theorem 4. More specifically, in \S\ref{sec:PosStuffOnC} we prove the second part of Theorem 4 and in \S\ref{Negative} we prove the first part.	In \S\ref{StructureTheorem} we prove how under certain conditions a twisted group algebra can be decomposed as a direct sum of twisted group algebra of a quotient group. This generalizes a known result for group algebras \cite[Theorem 3.2.9]{KarpilovskyProjective}. The main results of \S\ref{Rational} are Theorems 1 and 2, but we also include the description of twisted group algebras for elementary abelian $p$-groups of rank $2$. The main result of \S\ref{Dade} is Theorem 3.

\section{Preliminaries}\label{prelim}
	Throughout the whole paper $F$ denotes a field and its unit group is denoted by $F^*$. By $G$ we always denote a finite group. All the groups which will appear as the basis of a group algebra or twisted group algebra will be assumed to be finite. For a twisted group algebra $F^\alpha G$ we denote by $\{u_g \}_{g \in G}$ the basis corresponding to the elements of $G$ on which multiplication is explained by the cocycle $\alpha$. All cohomology groups are taken with respect to a trivial module structure. We often define a cohomology class $[\alpha] \in H^2(G, F^*)$ by giving the multiplication relations in $F^\alpha G$ where we frequently omit the relations in the basis $\{u_g \}_{g \in G}$ of $F^\alpha G$ which remain unchanged from $FG$. 
	
	We write $C_n$ for the cyclic group of order $n$. Moreover $M_n(F)$ denotes the $n \times n$-matrix ring over $F$. When $A$ is an $F$-algebra, $x \in A$ and $u$ a unit in $A$ we write $x^u = u^{-1}xu$.
	
	\subsection{The second cohomology group of a finite group and projective representations}\label{sec:H2}
We start by giving some basic results on the second cohomology group which we will need in the sequel. The results here are classical and contained in \cite{KarpilovskyProjective}.
We will use without additional explanation the known structure of the second cohomology group of abelian groups (see \cite{Schur} for the complex case and e.g. \cite[Corollary in \S 2.2]{YamazakiAbelian} for the general case).
	
The second cohomology group of a group $G$ over the complex numbers is denoted by $M(G)$ and is called the \emph{Schur multiplier}.
For an abelian group $A$ we denote by $Z^2(G,A)$ the $2$-cocycles of $G$ with values in $A$. For a cocycle $\alpha \in Z^2(G,A)$ we denote by $[\alpha]$ the cohomology class of $\alpha$ in $H^2(G,A)$. Moreover, $\alpha$ is \emph{symmetric} if $\alpha(x,y) = \alpha(y,x)$ for all $x,y \in G$. If $G$ is also abelian we set
\[\Ext(G,A) = \{[\alpha] \in H^2(G,A) \ | \ \alpha \ \text{is symmetric} \}. \]

An important tool to understand $H^2(G,F^*)$ is the following exact sequence.
\begin{theorem}\label{th:UCT}\cite[Theorem 11.5.2]{KarpilovskyVolII}
There is the following split short exact sequence
	\begin{equation*}
	1\rightarrow \operatorname{Ext}(G/G',F^*) \rightarrow
	H^2(G,F^*)\rightarrow \operatorname{Hom}(M(G),F^*)\rightarrow 1.
	\end{equation*}
\end{theorem}

	Let $n_1,\ldots,n_r$ be natural numbers. Recall that \cite[Corollary 2.3.17]{KarpilovskyProjective}
	\begin{equation}\label{eq:EXTdecomposition}
	\operatorname{Ext}(\Pi _{i=1}^rC_{n_r}, F^*) \cong \Pi _{i=1}^r
	\operatorname{Ext}(C_{n_r}, F^*).
	\end{equation}
	Therefore, in order to understand $\operatorname{Ext}(G/G',F^*)$
	it is sufficient to understand the description of
	$\operatorname{Ext}(C_n,F^*)\cong H^2(C_n,F^*)$. This is well
	known \cite[Theorem 1.3.1]{KarpilovskyProjective}:
	\begin{equation}\label{eq:cohoofcyclic}
	\operatorname{Ext}(C_n,F^*)\cong H^2(C_n,F^*)\cong F^*/(F^*)^n.
	\end{equation}

The map from $\Ext(G/G',F^*)$ to $H^2(G,F^*)$ in Theorem~\ref{th:UCT} is known as the \emph{inflation map} and we denote it as $\operatorname{inf}$. More general for a normal subgroup $N$ of $G$ and $\alpha \in Z^2(G/N, F^*)$ one has ${\operatorname{inf}}(\alpha)(x,y) = \alpha(xN, yN)$ for all $x,y \in G$ and this map implies a homomorphism from $H^2(G/N,F^*)$ to $H^2(G,F^*)$. 

If $V$ is an $F$-vector space, a map $\eta: G \rightarrow {\operatorname{GL}}(V)$, such that the composition of
$\eta$ with the natural projection from $GL(V)$ to $PGL(V)$ is a group homomorphism is called a \emph{projective representation} of $G$.
The projective representation $\eta$ is {\it irreducible}, if $V$ admits no proper $G$-invariant subspace. Two projective representations $\eta_1: G \rightarrow GL(V_1)$ and $\eta_2: G \rightarrow GL(V_2)$ are called \textit{projectively equivalent} if there is a map $\mu: G \rightarrow F^*$ satisfying $\mu(1) = 1$ and a vector space isomorphism $f:V_1 \rightarrow V_2$ such that
$$\eta_1(g) = \mu(g) f^{-1} \eta_2(g) f $$
for every $g \in G$.

If we define now a map $\alpha: G \times G \rightarrow F$ by 
$$\alpha(g_1,g_2)=\eta (g_1) \eta(g_2) \eta (g_1g_2)^{-1},$$
then $\alpha$ lies in $Z^2(G,F^*)$ and we refer to $\eta$ as an \emph{$\alpha$-representation} of $G$. For a fixed $\alpha \in Z^2(G,F^*)$, the set of projective equivalence classes of $\alpha$-representations of $G$ is denoted by $\text{Proj}(G, \alpha)$. As in the ordinary case, there is a natural correspondence between projective representations of $G$ over $F$ with an associated cohomology class $[\alpha]$ and $F^\alpha G$-modules. In particular, if the characteristic of $F$ does not divide the order of $G$, then $F^\alpha G$ is the direct sum of the simple $F^\alpha G$-modules which correspond to the projective equivalence classes of $\alpha$-representations. See \cite[Section 3.1]{KarpilovskyProjective} for more details.

We now record some facts on the Schur multiplier. We first recall that for any finite group $G$ there exists a \emph{Schur cover} $S_G$, which is also a finite group, such that $\C S_G \cong \oplus_{[\alpha] \in H^2(G,\C^*)} \C^\alpha G$.
	
Recall that a group action of a group $T$ on a group $N$ induces an action of $T$ on $M(N)$.
\begin{lemma} \cite[Section 2.2]{KarpilovskySchur}\label{lemma:semicoprimeSchur}
	Let $N$ and $T$ be subgroups of $G$ and assume $G = N \rtimes T$. Then $M(N)^T$ is a subgroup of $M(G)$. Moreover, if $N$ and $T$ are of coprime order, then
	\begin{equation*}
	M(G)=M(T)\times M(N)^T.
	\end{equation*}
	Here $M(N)^T$ consists precisely of those elements in $M(N)$ which are invariant under the $T$-action.
\end{lemma}

Set 
$$c_{G,\alpha}=\text{gcd} \{ \text{dim}(\eta) \ | \ \eta \in \text{Proj}(G,\alpha) \},$$
and denote, for a prime $p$, the biggest power of $p$ dividing $c_{G,\alpha}$ by $(c_{G,\alpha})_p$. 
\begin{proposition}\label{prop:Higgs88}\cite[Lemma 1 and Proposition 1]{Higgs1988}
	With the above notation
	\begin{enumerate}
		\item $\circ ([\alpha])$ divides $c_{G,\alpha}$.
		\item If a prime $p$ divides $c_{G,\alpha}$, then $p$ divides $\circ ([\alpha])$.
		\item  For a Sylow $p$-subgroup $G_p$ of $G$
		$$(c_{G,\alpha})_p=c_{G_p,\alpha}$$
		where on the right side $\alpha $ is the restriction of $\alpha $ to $G_p$.
	\end{enumerate}
\end{proposition}

Another important result when working over the complex numbers is sometimes referred to as ``Brauer's Trick".
\begin{proposition}\label{BrauerTrick}\cite[Theorem 2.3.2]{KarpilovskyProjective}
Let $\alpha \in Z^2(G,\C ^*)$. Then there exists a cocycle $\alpha' \in Z^2(G,\C ^*)$ which is cohomologous to $\alpha$ such that $\alpha'(g,h)$ is a root of unity of order dividing the order of $[\alpha]$ for any $g,h\in G$.
\end{proposition}

\subsection{Cyclic algebras and symbol algebras}\label{sec:Algebras}
We will need some basic facts about some well-known classes of central simple algebras. We fix here our notation for these algebras and cite the results on them which will be useful for our investigations. We refer to \cite[Chapter 15]{Pierce} and \cite[Chapter VII]{BO13} for the facts on cyclic algebras. We also prove a lemma of number theoretical nature which will be important later to disprove isomorphisms between twisted group algebras.

\begin{definition}\label{def:CyclicAlgebra} 
Let $K/F$ be a finite cyclic field extension of degree $n$ with Galois group generated by an element $\delta$ and let $\lambda \in F^*$. Let $A$ be the $K$-algebra which contains an element $a$ such that as a $K$-vector space we have $A = \oplus_{i=0}^{n-1} a^i K$ and for each $\mu \in K$ we have $\mu a = a\delta(\mu)$ and $a^n = \lambda$. Then $A$ is called a \emph{cyclic algebra} which we denote by $(\lambda,K/F,\delta)$.   
\end{definition}

We recall that two simple $F$-algebras $A$ and $B$ are called \emph{Brauer equivalent} if and only if there exists a division $F$-algebra $D$ and integers $n$ and $m$ such that $A \cong M_n(D)$ and $B \cong M_m(D)$. We write $A \equiv B$ to express Brauer equivalence. Moreover, we write $A \equiv 1$, if $A$ is isomorphic to a matrix ring over $F$. 

\begin{proposition}\label{prop:CyclicAlgebra}\cite[Proposition VII.1.9]{BO13}
The algebra $(\lambda,K/F,\delta)$ is a central simple $F$-algebra. If $L$ is a field such that $F \subseteq L \subseteq K$, then $(\lambda,K/F,\delta) \otimes_F L$ is Brauer equivalent to $(\lambda,K/L, \tilde{\delta})$ where $\tilde{\delta} = \delta^{[L:F]}$.
\end{proposition}

An especially nice form of cyclic algebras is the following.

\begin{definition}\cite[Remark 24.3]{Row08}
Let $n$ be a natural number and $F$ a field containing an $n$-th primitive root of unity $\zeta$. For $\lambda, \mu \in F^*$ define an $F$-algebra $A$ which is generated by two elements $x$ and $y$ and the relations $x^n = \lambda$, $y^n = \mu$ and $yx = \zeta xy$. Then this algebra is called a \emph{symbol algebra} and is denoted by $(\lambda, \mu; F; \zeta)_n$. When the field $F$ and the element $\zeta$ is clear from the context we simply write $(\lambda, \mu)_n$. 
\end{definition}

Note that the algebra $(\lambda, \mu; F; \zeta)_n$ can also be viewed as a twisted group algebra of $G \cong C_n\times C_n = \langle g\rangle \times \langle h \rangle$ over $F$. Namely, there is a cohomology class $[\alpha]\in H^2(G,F^*)$ determined by the following relations in $F^{\alpha}G$,
$$(u_g)^n = \lambda, \ (u_h)^n = \mu, \ u_hu_g = \zeta u_gu_h.  $$
It is easy to see that $F^\alpha G \cong (\lambda, \mu; F; \zeta)_n$.\\

\textbf{Notation:} Let $K/F$ be a finite field extension and denote by $\Nr(K/F)$ the norm map of this extension. We write $a \in \Nr(K/F)$, if $a$ is in the image of $\Nr(K/F)$. 

\begin{lemma}\label{lem:SymbolAlgebras}\cite[Proposition 24.48]{Row08}, \cite[Proposition VII.1.9]{BO13}\\
Let $F$ be a field containing a primitive $n$-th root of unity $\zeta$ and set $(\lambda, \mu)_n = (\lambda, \mu; F; \zeta)_n$ for any $\lambda, \mu \in F^*$. The following holds.
\begin{enumerate}
\item $(\lambda, 1)_n \equiv 1$.
\item $(\lambda, \mu)_n \equiv 1$ if and only if $\mu \in \Nr(F(\sqrt[n]{\lambda})/F)$.
\item $(\lambda, \mu_1)_n \otimes (\lambda, \mu_2)_n \equiv (\lambda, \mu_1 \mu_2)_n$.
\item $(\lambda, \mu)_n \cong (\mu, \lambda^{-1})_n$. 
\item $(\lambda, \mu_1) \cong (\lambda, \mu_2)$ if and only if $\mu_1 \mu_2^{-1} \in \Nr(F(\sqrt[n]{\lambda})/F)$.
\end{enumerate}
\end{lemma}

The preceding lemma shows that to decide on isomorphisms of symbol algebras one has to understand the images of the norm map of field extensions. Though in general this is not an easy task, at least for some situations important to us we can achieve this.
The proof of the following lemma is essentially due to Danny Neftin.
\begin{lemma}\label{lem:Norm}
Let $p$ be an odd prime, $\zeta$ a primitive complex $p$-th root of unity and set $F = \Q(\zeta)$. Then there exists $\lambda \in \Q^*$ such that $\zeta \notin \Nr(F(\sqrt[p]{\lambda})/F)$. 
\end{lemma}
\begin{proof}
For a fixed $\lambda$ set $z = \sqrt[p]{\lambda}$. 
Let $q$ be a prime such that $q \not\equiv 1 \bmod p$, which exists by Dirichlet's theorem. Denote by $\nu$ the $q$-adic evaluation of $F$ and by $F_\nu$ the completion of $F$ with respect to $\nu$. Note that by our choice of $q$ the field $F_\nu$ does not contain a root of unity of order $p^2$. By the Hasse Norm Theorem \cite[Section 18.4]{Pierce} to prove that $\zeta \notin \Nr(F(z)/F)$ it is sufficient to show that $\zeta \notin \Nr(F(z)_\nu/F_\nu)$. By a theorem of Albert \cite[Chapter 9, Theorem 11]{Albert37} this is equivalent to the condition that there exists a field $K$ such that $F_\nu \subseteq F(z)_\nu \subseteq K$ and $K/F_\nu$ is a cyclic extension of degree $p^2$. 

Assume now that $\lambda$ is a uniformizing element of $F_\nu$. Then the extension $F_\nu(z)/F_\nu$ is totally ramified, i.e. the ramification index equals $p$. Because there are exactly $p$ roots of unity in $F_\nu$ with order a power of $p$, it then follows from \cite[Proposition 4.1]{HankeSonn} that there is no field $K$ such that $F_\nu \subseteq F(z)_\nu \subseteq K$ and $K/F_\nu$ is a cyclic extension of degree $p^2$.
\end{proof}

\subsection{Partial isomorphisms of second cohomology groups}
We record here a fact which restricts the possible isomorphisms of the second cohomology groups which can come up in the (TGRIP). This fact will be particularly useful to us later when $H^2(G,F^*)$ will be infinitely generated. First recall the following result on commutative components of twisted group algebras \cite[Proposition 2.5]{MargolisSchnabel2}.

\begin{proposition}\label{prop:SalatimTheorem}
Let $F$ be a field such that ${\operatorname{char}}(F)$ does not divide the order of $G$ and let $[\alpha] \in H^2(G,F^*)$. Then $F^\alpha G$ admits a commutative Wedderburn component if and only if $[\alpha]$ is in the image of the inflation map from $\Ext(G/G',F^*)$ to $H^2(G,F^*)$.
\end{proposition}

\begin{proposition}\label{prop:IsosFromSalatim}
Let $G$ and $H$ be finite groups such that $G \sim_F H$. Assume moreover that the characteristic of $F$ does not divide the order of $G$. \\
Then $\Ext(G/G',F^*) \cong \Ext(H/H',F^*)$ and $\Hom(M(G),F^*) \cong \Hom(M(H),F^*)$.
\end{proposition}
\begin{proof}
Let $\varphi: H^2(G, F^*) \rightarrow H^2(H, F^*)$ be an isomorphism realizing the relation $G \sim_F H$.
Recall that 
\begin{equation}\label{eq:UCT}
H^2(G, F^*) \cong \Ext(G/G', F^*) \times \Hom(M(G), F^*)
\end{equation}
and $\Ext(G/G', F^*)$ can be identified with a subgroup of $H^2(G, F^*)$ via the inflation map. For $[\alpha] \in H^2(G, F^*)$ we know that $F^\alpha G$ admits a commutative simple Wedderburn component if and only if $[\alpha] \in \Ext(G/G', F^*)$ by Proposition~\ref{prop:SalatimTheorem}. So 
\begin{align*}
&[\alpha] \in\Ext(G/G', F^*) \Leftrightarrow F^\alpha G \ \ \text{admits a commutative component} \\
& \Leftrightarrow F^{\varphi(\alpha)} H \ \ \text{admits a commutative component} \Leftrightarrow \varphi([\alpha]) \in \Ext(H/H', F^*).
\end{align*}
This implies that $\varphi$ restricts to an isomorphism between the groups $\Ext(G/G', F^*)$ and $\Ext(H/H', F^*)$. In view of \eqref{eq:UCT} it then also induces an isomorphism between $\Hom(M(G), F^*)$ and $\Hom(M(H), F^*)$.
\end{proof}

\subsection{Other useful results to determine twisted group algebras}
For the convenience of the reader we include several more known results which we will use later to determine isomorphism types of twisted group algebras.
\begin{proposition}\label{prop:partialSchurcover}\cite[Proposition 3.3.8]{KarpilovskyProjective}
	Let $\alpha \in Z^2(G,\C ^*)$ be of order $n$ and let $G_{\alpha}=\langle \alpha(x,y) g \ | \ x,y,g\in G \rangle$. Then
	$$\C G_{\alpha}\cong \bigoplus_{i=0}^{n-1}\C ^{\alpha^i}G.$$
\end{proposition}

\begin{lemma}\label{lem:PassmanMatrices} \cite[Lemma 6.1.6]{Pas77} In a ring $R$ let $1 = e_1+\ldots+e_n$ be a decomposition into pairwise orthogonal idempotents. Assume $R^*$ permutes the set $\{e_1,\ldots,e_n \}$ transitively by conjugation. Then $R \cong M_n(e_1Re_1)$.
\end{lemma}

The following two results are due to It\^{o}.

\begin{theorem}\cite[(12.34) Corollary]{Isaacs}\label{th:Isaacs1234}
Let $G$ be solvable. Then $G$ has a normal abelian Sylow $p$-subgroup if and only if $p$ does not divide the degree of any complex irreducible character of $G$. 
\end{theorem}

\begin{theorem}\cite[(6.15) Theorem]{Isaacs}\label{th:Isaacs615}
Let $A$ be a normal abelian subgroup of $G$. Then the degree of any irreducible complex character of $G$ divides $[G:A]$.
\end{theorem}

\section{Powering of elements in the Schur multiplier}\label{Section3}
The following theorem will be an important tool in the proof of Theorem 4.
The proof presented here is due to Y. Ginosar. Recall that for $[\alpha] \in H^2(G,\C^*)$ we denote by $c_{G,\alpha}$ the greatest common divisor of the dimensions of the irreducible $\alpha$-representations.

\begin{theorem}\label{th:isomiffcoprime}
Let $G$ be a finite group, let $[\alpha] \in M(G)$ be of order $n$ and let $r$ be a natural number.
If $r$ is coprime with $n$, then $\mathbb{C}^\alpha G \cong \C^{\alpha^r} G$ and if $r$ is not coprime with $n$, then $c_{G,\alpha} > c_{G,\alpha^r}$.
In particular, $\C ^{\alpha}G\cong \C ^{\alpha^r}G$ if and only if $\gcd(n,r) = 1$.
\end{theorem} 
\begin{proof}
Let $\zeta$ be a primitive $n$-th root of unity in $\C$.
Assume first $\gcd(n,r)=1$. Using Proposition~\ref{BrauerTrick} we
may assume that the values of $\alpha$ in $\C ^*$ are all powers of $\zeta$.  
Since $\gcd(n,r)=1$, the following is an isomorphism
$$
\begin{array}{rl}
\psi :\mathbb{Q}(\zeta) & \rightarrow \mathbb{Q}(\zeta)\\
\zeta & \mapsto \zeta ^r.
\end{array}
$$
Let
$$\C ^{\alpha}G=\text{span}_{\C}\{u_g\}_{g\in G} \text{ and } \C ^{\alpha ^r}G=\text{span}_{\C}\{v_g\}_{g\in G}.$$
Then $\psi$ can be extended to an automorphism of $ \C$ which can
be extended as follows
$$
\begin{array}{rl}
\tilde{\psi} :\C ^{\alpha}G & \rightarrow \C ^{\alpha ^r}G\\
\sum _{g\in G} a_gu_g & \mapsto \sum _{g\in G} \psi(a_g)v_g.
\end{array}
$$
In order to show that $\tilde{\psi}$ is a ring isomorphism it is
enough to check that
$$\tilde{\psi}(u_gu_h)=\tilde{\psi}(\alpha (g,h)u_{gh})=\psi(\alpha (g,h))v_{gh}=\alpha^r (g,h)v_{gh}=v_gv_h=\tilde{\psi}(u_g)\tilde{\psi}(u_h).$$

Next, we assume that $\gcd(n,r)=d$ where $d \neq 1$.
Let $s=\frac{r}{d}$. Then, $r=sd$ and  $\gcd(s,n)=1$. Consequently 
$$\mathbb{C} ^{\alpha ^r}G=\mathbb{C} ^{(\alpha ^d)^s}G\cong \mathbb{C} ^{\alpha ^d}G,$$
where the right isomorphism derives from the fact that the order of $[\alpha]^d$ is prime to $s$.
Hence we may assume that $r$ is a divisor of $n$.	

We will reduce the problem to $r$ being a prime divisor of $n$. Let $p$ be a prime divisor of $r$. Denote by $G_p$ a Sylow $p$-subgroup of $G$. By Proposition~\ref{prop:Higgs88} in order to show that $c_{G,\alpha} > c_{G,\alpha ^r}$ it is sufficient to show that $c_{G_p,\alpha}$ is greater than $c_{G_p,\alpha^r}$.
Note that since the $p$-part of $M(G)$ is embedded inside $M(G_p)$ \cite[Corollary 2.3.24 and its proof]{KarpilovskyProjective} 
we know that $[\alpha]$ is not trivial as an element of $M(G_p)$. 
Hence it will be sufficient to assume that the order of $[\alpha] \in M(G_p)$ is a power of $p$. In order to finish the proof we will show that 
$c_{G_p,\alpha}$ is greater than $c_{G_p,\alpha^p}$. 

Since $p$ is a divisor of the order of $[\alpha]$, the dimension of any element in $\text{Proj}(G_p,\alpha)$ is divisible by $p$, cf. Proposition~\ref{prop:Higgs88}. Therefore the dimension of any irreducible $\alpha$-representation can be written as $p^j$ such that $j\geq 1$ \cite[Main Theorem]{Higgs1988}. Hence $c_{G_p, \alpha}$ is a $p$-power and there is an irreducible $\alpha$-representation of dimension $c_{G_p,\alpha}$. Let $\eta \in \text{Proj}(G_p,\alpha)$ be such a representation.
Now, denote by $\gamma$ the $p$-exterior power of $\eta$ with itself, cf. \cite[Section 9.8]{Rotman} for the definition. Then, $\gamma \in \text{Proj}(G_p,\alpha^p)$ \cite[Section 2.2]{GinosarSchnabel}, and additionally, by \cite[Theorem 9.140]{Rotman},
$$\text{dim}(\gamma)=\binom{c_{G_p,\alpha}}{p},$$ 
and hence $\text{dim}(\gamma)$ is divisible by $\frac{c_{G_p, \alpha}}{p}$. Thus $c_{G_p,\alpha}$ is greater than $c_{G_p,\alpha^p}$. 
\end{proof}

\begin{remark}
It would be interesting to know if a result like Theorem~\ref{th:isomiffcoprime} holds for all fields, i.e. whether $F^\alpha G \cong F^{\alpha^r} G$ if and only if $r$ is coprime with the order of $[\alpha] \in H^2(G,F^*)$. We are not aware of a counterexample, but it is also clear that the proof given above can not be directly translated to the general situation. 
\end{remark}

\section{Proof of Theorem 4}\label{sec:Proof of Theorem 4}
In this section we will prove Theorem 4 and give some additional results concerning sufficient conditions for $G\sim _F H$.
We split this section into two parts. In the first part we give some positive results for Question~\ref{Q:suff_conditions} (ii) and in the second part we give a negative answer to Question~\ref{Q:suff_conditions} (i).

\subsection{Positive results for Question~\ref{Q:suff_conditions} (ii)}\label{sec:PosStuffOnC}
For a group $G$, denote by $S_G$ a Schur cover of $G$.
Then, the following is an immediate corollary of Theorem~\ref{th:isomiffcoprime}.
\begin{corollary}\label{cor:cohomologyC_p}
	Let $G$ and $H$ be groups such that $\C G\cong \C H$, $\C S_G \cong \C S_H$ and $M(G) \cong M(H) \cong C_p$. Then $G \sim _{\C} H$.
\end{corollary}
\begin{proof}
	Let $[\alpha]$ and $[\beta]$ be generators of $M(G)$ and $M(H)$ respectively. Then, since $\C G\cong \C H$, we have
	$$\bigoplus _{i=1}^{p-1}\C ^{\alpha ^i}G\cong \bigoplus _{i=1}^{p-1}\C ^{\beta ^i} H.$$ 
	The result now follows from Theorem~\ref{th:isomiffcoprime} since $\C ^{\alpha ^i}G\cong \C ^{\alpha ^j}G$ and $\C ^{\beta ^i}G\cong \C ^{\beta ^j}G$ for any $1\leq i, j \leq p-1$.
\end{proof}

The following is an elementary observation.

\begin{lemma}
Let $F$ be a field and let $G$ and $H$ be groups satisfying the $F$-bijectivity condition and in addition $H^2(G,F^*)\cong H^2(H,F^*)\cong C_2\times C_2$. Then $G\sim _F H$.
\end{lemma}
\begin{proof}
The claim follows immediately from the fact that any set bijection from $C_2\times C_2$ to itself fixing the trivial element is a group automorphism.
\end{proof}

We will prove a similar lemma for $H^2(G,F^*)\cong H^2(H,F^*)\cong C_3\times C_3$, but only for $F$ being the complex numbers.

\begin{lemma}
	Let $G$ and $H$ be groups satisfying the $\C$-bijectivity condition and in addition $M(G) \cong M(H) \cong C_3 \times C_3$. Then $G\sim _\C H$.
\end{lemma}
\begin{proof}
Denote 
$$M(G)=\langle \alpha \rangle \times \langle \beta \rangle,\quad M(H)=\langle \tilde{\alpha} \rangle \times \langle \tilde{\beta}\rangle.$$
By Theorem~\ref{th:isomiffcoprime} 
$$\C ^{\alpha}G\cong \C ^{\alpha ^2}G,\quad \C ^{\beta}G\cong \C ^{\beta ^2}G,\quad \C ^{\alpha \beta}G\cong \C ^{\alpha ^2\beta ^2}G,\quad \C ^{\alpha ^2\beta}G\cong \C ^{\alpha  \beta^2}G,$$
and similarly
$$\C ^{\tilde{\alpha}}H \cong \C ^{\tilde{\alpha} ^2}H,\quad \C ^{\tilde{\beta}}H \cong \C ^{\tilde{\beta}^2}H,\quad \C ^{\tilde{\alpha} \tilde{\beta}}H \cong \C ^{\tilde{\alpha}^2\tilde{\beta}^2}H,\quad \C ^{\tilde{\alpha}^2\tilde{\beta}}H \cong \C^{\tilde{\alpha}  \tilde{\beta^2}}H.$$
So, the map $\psi $ providing the $\C$-bijectivity is a permutation on these (at most) four isomorphism types. The action of $\Aut(C_3 \times C_3)$  on the non-trivial cyclic subgroups of $C_3 \times C_3$ is given by $\Aut(C_3\times C_3)/Z(\Aut(C_3\times C_3)) \cong \operatorname{PGL}(2,3) \cong S_4$. Hence any permutation of the non-trivial cyclic subgroups of $C_3 \times C_3$ can be realized by an automorphism which implies that any permutation of the isomorphism types of the twisted group algebras can be realized by an isomorphism between the Schur multipliers.
\end{proof}

We do not have a similar proof for the case that the Schur multiplier is $C_p\times C_p$ for $p \geq 5$. Clearly the proofs above do not work. However, we do neither have a counterexample in this case.

Maybe the most interesting positive result on Question~\ref{Q:suff_conditions} is the following.
\begin{theorem}\label{th:cyclicSchur}
	Let $G$ and $H$ be groups satisfying the $\C$-bijectivity condition and in addition $M(G) \cong M(H) \cong C_n$. Then $G\sim_{\mathbb{C}} H$.
\end{theorem}
\begin{proof}
Denote the map which corresponds to the $\C$-bijectivity by $\psi :M(G)\rightarrow M(H)$.
Let $[\alpha]$ be a generator of $M(G)$. We will prove that $\psi ([\alpha])=[\beta]$ is a generator of $M(H)$. Assume $[\beta]$ is of order $k$, a divisor of $n$. Owing to Theorem~\ref{th:isomiffcoprime}, $\varphi (n)=\varphi (k)$ where $\varphi$ denotes Euler's totient function. Since $k$ is a divisor of $n$ we get that either $k=n$ or $n=2k$ where $k$ is an odd integer.

Assume the latter. 
%
Then there exists a generator $[\gamma]\in M(H)$ such that $[\beta]=[\gamma]^2$, implying $c_{H,\beta} < c_{H,\gamma}$ by Theorem~\ref{th:isomiffcoprime}. On the other hand $\mathbb{C}^\gamma H \cong \mathbb{C}^{\alpha^r}$ for some natural number $r$, implying $c_{H, \gamma} \leq c_{G, \alpha}$, also by Theorem~\ref{th:isomiffcoprime}. Hence $c_{H,\beta} < c_{G,\alpha}$, which contradicts $\mathbb{C}^\alpha G \cong \mathbb{C}^\beta H$. We proved that $\psi$ sends a generator of $M(G)$ to a generator of $M(H)$. From here it is straightforward to prove $G\sim_{\mathbb{C}} H$ using Theorem~\ref{th:isomiffcoprime}.
\end{proof}

\subsection{Negative answer for Question~\ref{Q:suff_conditions} (i)}\label{Negative}
Throughout this section $G$ and $H$ will be the following groups, which we introduce with abuse of notation (we use the same names for generators of the groups). The identities of these groups in the SmallGroupLibrary of GAP \cite{GAP} are $\texttt{[288, 484]}$ and $\texttt{[288, 603]}$, but we will only employ theoretical arguments. 
The groups $G$ and $H$ are generated by elements $a,b,c,d,e$ such that 
\[a^3=b^3=c^4=d^4=e^2=1.\]
The Sylow $3$-subgroup in both groups is $\langle a \rangle \times \langle b \rangle \cong C_3 \times C_3$. A Sylow $2$-subgroup in both groups is given by 

\[(\langle c \rangle \times \langle d \rangle) \rtimes \langle e \rangle \ \text{where} \ [c,e] = d^2 \ \text{and} \ [d,e] = 1.\]
Both groups can be written as semidirect products of the form
\[ (\langle a \rangle \times \langle b \rangle) \rtimes ((\langle c \rangle \times \langle d \rangle) \rtimes \langle e \rangle) \]
where in both groups the following relations hold
\[[a,c] = 1, [a,d] = a, [b,d] = 1 \]
and
\begin{align*}
&\text{in} \ G: \ [a,e] =a, \ [b,c] = 1, \ [b,e]=b, \\
&\text{in} \ H: \ [a,e] =1, \ [b,c] = b, \ [b,e]=1.
\end{align*}

We will prove that $G$ and $H$ satisfy $M(G) \cong M(H)$ and the $\C$-bijectivity condition, but $G\not \sim_{\mathbb{C}} H$. 
In both groups the commutator subgroup is
$$\langle a \rangle \times \langle b \rangle \times \langle d ^2 \rangle \cong C_3\times C_3\times C_2.$$ 
We start with computing the Schur multipliers of these groups. 
Consider both groups as semi-direct products of the normal subgroup $N=\langle a \rangle \times \langle b \rangle$ and the subgroup $T=\left(\langle c \rangle \times \langle d \rangle \right) \rtimes \langle e \rangle.$
We have $M(C_3 \times C_3) \cong C_3$. So $M(N)$ is cyclic of order $3$, a generator being determined by the relation $[u_a,u_b]=\zeta _3$ in the twisted group algebra, where $\zeta_3$ denotes a primitive $3$-rd root of unity in $\C$. However this cohomology class is not invariant under the action of $T$.
Indeed
$$[u_{d(a)},u_{d(b)}]=[u_{a^2},u_b]=\zeta _3^2\neq [u_a,u_b].
$$
 
 Consequently, by Lemma~\ref{lemma:semicoprimeSchur}
$$M(G)\cong M(T)\cong M(H).$$

Now consider $N_1 = \langle c\rangle \times \langle d \rangle \cong C_4\times C_4$ and $T_1 = \langle e \rangle \cong C_2$. Then $T=N_1\rtimes T_1$.
Clearly $M(N_1)\cong C_4$, a generator being determined by the relation $[u_c,u_d]=\zeta _4$ in the twisted group algebra, where $\zeta_4$ denotes a primitive $4$-th root of unity in $\C$. It is easy to check that this cohomology class is invariant under the $T_1$-action. Another nontrivial cohomology class of $T$ is determined by $[u_d,u_e]=-1$ and it is easy to show that these two non-trivial cohomology classes generate $M(G)\cong M(H)$.

\begin{corollary}\label{cor:SchurC4XC2}
	With the above notation 
	$$M(G)\cong M(H)\cong C_2\times C_4=\langle \alpha \rangle \times \langle \beta \rangle,$$
	where 
	$$\alpha :\ [u_d,u_e]=-1,\quad \beta :\ [u_c,u_d]=\zeta _4.$$
\end{corollary}

Notice that it is not surprising, by the results of Section~\ref{sec:PosStuffOnC}, that in the negative example we found to Question~\ref{Q:suff_conditions} (i) the groups admit a Schur multiplier isomorphic to $C_4\times C_2$.

We now start computing the Wedderburn decomposition of the twisted group algebras of these groups.
We will start with the group algebras.

\begin{lemma}\label{lem:groupalgebra}
	With the above notation 
	$$\C G\cong \C H\cong 16 \C \oplus 28 \C^{2\times 2}\oplus 10 \C^{4\times 4}.$$
\end{lemma}
\begin{proof}
	First, notice that by Theorem~\ref{th:Isaacs1234} the degrees of all the characters of $G$ and $H$ are powers of $2$.
	Notice that both groups $G$ and $H$ admit a normal abelian subgroup of index $4$. For $G$ we can choose $\langle a,b,c,d^2\rangle$ and for $H$ we can choose $\langle a,b,c^2,d^2,e\rangle$. Hence, it follows from Theorem~\ref{th:Isaacs615} that the maximal degree of an irreducible character of $G$ and $H$ is smaller or equal to $4$. Consequently, possible dimension of the irreducible representations of these groups are $1,2$ or $4$.
	Next, since $|G/G'|=16=|H/H'|$ we have that $\mathbb{C}G$ and $\mathbb{C}H$ admit exactly $16$ copies of the field $\C$ in their Wedderburn decomposition. 
	To finish the proof observe that both these groups admit $54$ conjugacy classes, and hence the dimension of the center of
	$\mathbb{C}G$ and $\mathbb{C}H$ is $54$. 
\end{proof}

In order to compute the Wedderburn decomposition for the twisted group algebras we will use Proposition~\ref{prop:partialSchurcover}.

\begin{lemma}\label{lem:betasquare}
	With the above notation, 
	$$\C^{\beta ^2}G\cong 32\C ^{2\times 2}\oplus 10 \C ^{4\times 4}
	\text{ and } \C^{\beta ^2}H \cong 16\C ^{2\times 2}\oplus 14 \C ^{4\times 4}.$$
\end{lemma}
\begin{proof}
	We will use Proposition~\ref{prop:partialSchurcover}. 
	By the construction of $G_{\omega}$, we get that 
	$$G_{\beta ^2}\cong \left(\overset{a}C_3\times \overset{b}C_3\right)\rtimes \left((\overset{f}C_2\times \overset{c}C_4)\rtimes \overset{d}C_4)\rtimes \overset{e}C_2 \right),$$
	where $f$ is central and $[c,d]=f$ and all the other relations are the same as in $G$.
	Similarly,
	$$H_{\beta ^2}\cong \left(\overset{a}C_3\times \overset{b}C_3\right)\rtimes \left((\overset{f}C_2\times \overset{c}C_4)\rtimes \overset{d}C_4)\rtimes \overset{e}C_2 \right),$$
	where $f$ is central and $[c,d]=f$ and all the other relations are the same as in $H$.
	Then, again, as in Lemma~\ref{lem:groupalgebra}, by Theorem~\ref{th:Isaacs1234} the degrees of all the characters of $G_{\beta ^2}$ and $H_{\beta ^2}$ are powers of $2$. And also, both groups  $G_{\beta ^2}$ and  $H_{\beta ^2}$ admit a normal abelian subgroup of index $4$. For  $G_{\beta ^2}$ we can choose $\langle a,b,c,d^2,f\rangle$ and for  $H_{\beta ^2}$ we can choose $\langle a,b,c^2,d^2,e,f\rangle$. Hence, it follows from Theorem~\ref{th:Isaacs615} that the maximal degree of an irreducible character of  $G_{\beta ^2}$ and  $H_{\beta ^2}$ is smaller or equal to $4$.
	A simple calculation shows that $G_{\beta ^2}$ admits $96$ conjugacy classes while  $H_{\beta ^2}$ admits $84$ conjugacy classes. It follows that
	$$\C G_{\beta ^2}\cong 16\C \oplus 60\C ^{2\times 2}\oplus 20\C ^{4\times 4}$$
	and
	$$\C H_{\beta ^2}\cong 16\C \oplus 44\C ^{2\times 2}\oplus 24\C ^{4\times 4}.$$
	The result follows by Lemma~\ref{lem:groupalgebra} and by noticing that 
	$$\C G_{\beta ^2} \cong \C G\oplus \C ^{\beta ^2}G \text{ and }
	\C H_{\beta ^2} \cong \C H\oplus \C ^{\beta ^2}H.$$ 
\end{proof}

\begin{corollary}\label{cor:notinrelation}
	With the above notation $G\not \sim_{\mathbb{C}} H$.
\end{corollary}
\begin{proof}
Assume that $\psi: M(G) \rightarrow M(H)$ is an isomorphism. As the squares in $M(G)$ and $M(H)$ form a characteristic subgroup, $\psi ([\beta^2]) = [\beta^2]$. Hence by Lemma~\ref{lem:betasquare} we can not have an isomorphism $\C^\gamma G \cong \C^{\psi(\gamma)} H$ for all $[\gamma] \in M(G)$. 
\end{proof}

In the rest of this section we will prove that the groups $G$ and $H$ satisfy the $\C$-bijectivity condition.
\begin{lemma}\label{lem:cohomologyoforder4}
	Let $[\gamma] \in\{[\beta] ,[\beta ^3], [\alpha \beta] , [\alpha \beta ^3]\}$ in $M(G)$ or alternatively $M(H)$. Then
	$$\C ^{\gamma}G\cong \C ^{\gamma}H\cong 6\C ^{4\times 4}\oplus 3\C ^{8\times 8}.$$
\end{lemma}
\begin{proof}
	We will use Proposition~\ref{prop:partialSchurcover}. 
	By the construction of $G_{\omega}$, we get that 
	$$G_{\beta }\cong \left(\overset{a}C_3\times \overset{b}C_3\right)\rtimes \left((\overset{f}C_4\times \overset{c}C_4)\rtimes \overset{d}C_4)\rtimes \overset{e}C_2 \right),$$
	where $f$ is central and $[c,d]=f$ and all the other relations are the same as in $G$.
	Similarly,
	$$H_{\beta}\cong \left(\overset{a}C_3\times \overset{b}C_3\right)\rtimes \left((\overset{f}C_4\times \overset{c}C_4)\rtimes \overset{d}C_4)\rtimes \overset{e}C_2 \right),$$
	where $f$ is central and $[c,d]=f$ and all the other relations are the same as in $H$.
	To prove the lemma we can argue as in the proofs of Lemma~\ref{lem:groupalgebra} and Lemma~\ref{lem:betasquare} using that the degree of any $\beta$-representation is divisible by $4$, that $\langle a,b,c^2,d^2,f\rangle$ is an abelian normal subgroup in both $G$ and $H$ and that $\C^\beta G \cong \C^{\beta^3}G$ and $\C^\beta H \cong \C^{\beta^3} H$. Analogues arguments also handle the cohomology classes $[\alpha\beta]$ and $[\alpha\beta^3]$.
\end{proof}

We are left with the cohomology classes $[\alpha]$ and $[\alpha \beta^2]$ in both $M(G)$ and $M(H)$.
Using similar techniques to the above we can show the following. 
\begin{lemma}\label{lem:cohomologyoforder2}
	With the above notation
$$\C ^{\alpha}G\cong \C ^{\alpha \beta ^2}G \cong 16\C ^{2\times 2}\oplus 14 \C ^{4\times 4},$$
$$\C ^{\alpha}H \cong 16\C ^{2\times 2}\oplus 14 \C ^{4\times 4}$$
and
$$\C ^{\alpha \beta ^2}H \cong 32\C ^{2\times 2}\oplus 10 \C ^{4\times 4}.$$
\end{lemma}

\begin{corollary}\label{cor:ConditionBsatisf}
$G$ and $H$ satisfy the $\C$-bijectivity condition.
\end{corollary}
\begin{proof}
We can take $\varphi :M(G)\rightarrow M(H)$ to be determined by
\begin{enumerate}
	\item $\varphi ([\gamma])=[\gamma]$ for $[\gamma]\in \{1, [\alpha], [\beta], [\beta ^3] , [\alpha \beta], [\alpha \beta ^3]\}$.
	\item $\varphi ([\beta ^2]) = [\alpha \beta ^2]$.
	\item $\varphi ([\alpha \beta ^2])= [\beta ^2]$.
\end{enumerate}
Then, the result follows from Lemmas~\ref{lem:groupalgebra}-\ref{lem:cohomologyoforder2}.
\end{proof}

Notice now that Theorem 4.(2) follows from Corollary~\ref{cor:SchurC4XC2}, Corollary~\ref{cor:notinrelation} and Corollary~\ref{cor:ConditionBsatisf}.

\section{Structure theorem for twisted group algebras of central extensions}\label{StructureTheorem}
The aim of this section is to generalize a result \cite[Theorem 3.2.9]{KarpilovskyProjective} which describes the decomposition of a group algebra of a group $S$ into a direct sum of twisted group algebras of a quotient of $S$ by a central subgroup. Note that it is known that $F^\alpha G$ only maps naturally to $F^\beta G/N$, for $N$ some normal subgroup of $G$ and $[\beta] \in H^2(G/N,F^*)$, if $[\alpha]$ is in the image of the inflation map from $H^2(G/N,F^*)$ to $H^2(G,F^*)$ \cite[Lemma 3.2.12]{KarpilovskyProjective}. 

We first introduce some notation.
Let
\begin{equation}\label{eq:CentrExt}
 1\rightarrow A \rightarrow S \rightarrow G \rightarrow 1
\end{equation}
be a central extension.
Let $\alpha'\in Z^2(G,F^*)$ and let 
$\alpha \in Z^2(S,F^*)$ such that $\alpha$ is inflated from $\alpha'$.
Now, fix a 2-cocycle $\sigma \in Z^2(G,A)$ such that $[\sigma] \in H^2(G,A)$ corresponds to the central extension in \eqref{eq:CentrExt}. Note that in this way we also fix a section $G \rightarrow S$. Define for any  field $K$ containing $F$ and $\chi \in \Hom(A,K^*)$ the following function 
$$\alpha_{\chi}:G\times G \rightarrow K^*, \ \alpha_{\chi} (g,h)=\alpha'(g,h)\chi (\sigma (g,h)).$$

\begin{lemma}
With the above notation $\alpha_{\chi}\in Z^2(G,K^*)$.
\end{lemma}
\begin{proof}
It is straightforward to verify the 2-cocyle condition for $\alpha_\chi$ using that $\alpha'$ and $\sigma$ are 2-cocycles and $\chi$ a group homomorphism from $A$ to $K^*$.
\end{proof}

Denote by $F(\chi)$ the smallest field extension of $F$ which contains all the values of $\chi$. Note that then $\alpha_\chi$ has only values in $F(\chi)^*$, so $\alpha_\chi \in Z^2(G,F(\chi)^*)$. We view $G$ as a subset of $S$ via the section given by $\sigma$.

\begin{proposition}\label{prop:SurjectiveRingHomom}
 Let $\{u_s\}_{s\in S}$ be a basis of $F^{\alpha}S$ corresponding to $\alpha$ and 
$\{v_g\}_{g\in G}$ a basis of $F(\chi)^{\alpha_{\chi}}G$ corresponding to $\alpha_{\chi}$. Define 
$$\psi_\chi :F^{\alpha}S\rightarrow F(\chi)^{\alpha_{\chi}}G, \ \ k u_{ag} \mapsto k \chi (a)v_g $$
for any $k \in F$, $a \in A$ and $g \in G$ and by extending $\psi_\chi$ linearly to all of $F^\alpha S$.
Then $\psi_\chi$ is a surjective ring homomorphism.
\end{proposition}
\begin{proof}
We set $\psi = \psi_\chi$. Since $F$ is central in $F^{\alpha}S$ and $F(\chi)$ is central in $F(\chi)^{\alpha_{\chi}}G$ it is sufficient to prove that
$$\psi (u_{a_1g_1}u_{a_2g_2})=\psi (u_{a_1g_1})\psi (u_{a_2g_2})$$
 for any $g_1, g_2 \in G$ and $a_1, a_2 \in A$.
Notice that
$$u_{a_1g_1}u_{a_2g_2}=\alpha (a_1g_1,a_2g_2)u_{a_1g_1a_2g_2}=\alpha (a_1g_1,a_2g_2)u_{a_1a_2\sigma(g_1,g_2)g_1g_2}.$$ 
Therefore,
$$\psi (u_{a_1g_1}u_{a_2g_2})=\alpha (a_1g_1,a_2g_2)\chi (a_1a_2\sigma(g_1,g_2))v_{g_1g_2}.$$
Using the fact that $\alpha$ is inflated from $\alpha'$ and the definition of $\alpha_{\chi}$ we get that 
$$\psi (u_{a_1g_1}u_{a_2g_2})=\alpha' (g_1,g_2)\chi (a_1a_2\sigma(g_1,g_2))v_{g_1g_2}=\chi (a_1a_2)\alpha_{\chi}(g_1,g_2)v_{g_1g_2}.$$
On the other hand,
$$\psi (u_{a_1g_1})\psi (u_{a_2g_2})=\chi(a_1)v_{g_1}\chi(a_2)v_{g_2}=\chi (a_1a_2)v_{g_1}v_{g_2}=\chi (a_1a_2)\alpha_{\chi}(g_1,g_2)v_{g_1g_2}.$$	
We proved that $\psi$ is indeed a homomorphism. The surjectivity of $\psi $ is clear since $v_g$ is in the image of $\psi$ for any $g\in G$ and from the definition of $F(\chi)$.
\end{proof}

We now proceed to generalize \cite[Theorem 3.2.9]{KarpilovskyProjective} for cohomology classes of $S$ which are inflated from cohomology classes of $G$. 
Assume that the characteristic of $F$ does not divide the order of $A$, let $K$ be a field containing $F$ and a primitive $\exp(A)$-th root of unity. Denote by $\mathcal{M}$ the set of simple $FA$-modules up to isomorphism. Note that as $A$ is abelian we can identify $\Hom(A,K^*)$ with the irreducible $K$-characters of $A$. 
Let $\Irr(A, F)$ be a subset of $K$-characters of $A$ such that it contains for each $M \in \mathcal{M}$ one character which corresponds to a composition factor of $K \otimes_F M$. In particular, $|\mathcal{M}| = |\Irr(A,F)|$. 

\begin{theorem}\label{th:NEWINFSTUFF}
Let $1 \rightarrow A \rightarrow S \rightarrow G \rightarrow 1$ be a central extension of finite groups, $F$ a field of characteristic not dividing the order of $A$ and $[\alpha] \in H^2(S,F^*)$ a cohomology class inflated from an element in $H^2(G,F^*)$. Then 
\begin{align*}
F^{\alpha}S\cong \bigoplus _{\chi \in \Irr(A,F)}F(\chi)^{\alpha_{\chi}}G.
\end{align*}	
\end{theorem}
\begin{proof}
Define
\[\psi: F^\alpha S \rightarrow \bigoplus _{\chi \in \Irr(A,F)}F(\chi)^{\alpha_{\chi}}G, \ \ x \mapsto (\psi_\chi(x))_{\chi \in \Irr(A,F)} \]
where $\psi_\chi$ is defined as in Proposition~\ref{prop:SurjectiveRingHomom} for every $\chi \in \Irr(A,F)$ and $(\psi_\chi(x))_{\chi \in \Irr(A,F)}$ denotes a tuple with entries indexed by the elements of $\Irr(A,F)$. By Proposition~\ref{prop:SurjectiveRingHomom} the map $\psi$ is a surjective ring homomorphism on every direct summand and as the dimensions of $F^{\alpha}S$ and  $\bigoplus_{\chi \in \Irr(A,F)}F(\chi)^{\alpha_{\chi}}G$ are equal, we are done once we show that $\psi$ is injective.

We view $G$ embedded in $S$, as a transversal of the cosets of $A$. Then a general element $x \in F^\alpha S$ can be written as $x = \sum_{g \in G}a_g g$ where $a_g \in FA$. Fix some $\chi \in \Irr(A,F)$. Then, by the definition of $\psi_\chi$, the element $x$ lies in the kernel of $\psi_\chi$ if and only if $\chi(a_g) = 0$ for all $g \in G$. Hence $x$ lies in the in the kernel of $\psi$ if and only if $\chi(a_g) = 0$ for all $g \in G$ and all $\chi \in \Irr(A,F)$. By our assumption on $F$ the algebra $FA$ is semisimple, so $FA$ is isomorphic to the direct sum of the elements in $\mathcal{M}$, the simple $FA$-modules. Thus by the definition of $\Irr(A,F)$ we have, for a fixed $g \in G$, that $\chi(a_g) = 0$ for all $\chi \in \Irr(A,F)$ if and only if $a_g$ annihilates all simple $FA$-modules, i.e. $a_g = 0$. Hence $\chi(a_g) = 0$ for all all $g \in G$ and $\chi \in \Irr(A,F)$ if and only if $a_g = 0$ for all $g \in G$, i.e. $x = 0$. Consequently $\ker(\psi) = 0$ and $\psi$ is injective.  
\end{proof}

 \section{Twisted group algebras of $p$-groups}\label{Rational}
In this section we study the twisted group algebras of $p$-groups starting with elementary abelian groups of rank $2$. We then proceed to prove Theorems 1 and 2.

Throughout this section $p$ denotes a prime and $F$ a field of characteristic different from $p$. Moreover $\zeta$ denotes a primitive $p$-th root of unity in $F$ or an extension of $F$. Let $[F(\zeta):F] = \frac{p-1}{d}$ for some divisor $d$ of $p-1$ and let $\zeta_1\ldots,\zeta_d$ be representatives of the orbits of primitive $p$-th roots of unity in $F(\zeta)$ under the action of $\Aut(F(\zeta)/F)$. Note that $F(\zeta_i) = F(\zeta_j)$ for any $1 \leq i,j \leq d$. We will use Theorem~\ref{th:UCT} throughout this section without mentioning it explicitly.

\subsection{Twisted group algebras of elementary abelian groups of rank 2}\label{sec:CpxCp}
We start by investigating the twisted group algebras of elementary abelian groups of rank $2$. This example will turn out to be crucial in the sequel. 

Let $G = \langle a \rangle \times \langle b \rangle \cong  C_p \times C_p$. 
The Schur multiplier of $G$ is isomorphic to $C_p$, a generator being explained by the relation $[u_a,u_b] = \zeta$. Therefore, if $\zeta \in F$, then $\Hom(M(G), F^*) \cong C_p$ and $\Hom(M(G), F^*) = 1$ otherwise. Moreover, $\Ext(G/G', F^*) = \Ext(G, F^*)$, so by Section~\ref{sec:H2} a generic cohomology class $[\alpha] \in H^2(G, F^*)$ is determined by parameters $\lambda, \mu \in F^*$ and an integer $0 \leq \ell \leq p-1$ with corresponding relations in the twisted group algebra given by 
$$u_a^p = \lambda, \ u_b^p = \mu, \ [u_a, u_b] = \zeta^\ell $$
where $\ell = 0$ if $\zeta \notin F$.

Assuming $\ell \neq 0$, in particular $\zeta \in F$, the relations in $F^\alpha G$ are exactly the defining relations of the symbol algebra $(\lambda, \mu;F;\zeta^\ell)_p$. Hence, $F^\alpha G$ is a central simple $F$-algebra which is by Lemma~\ref{lem:SymbolAlgebras} isomorphic to a matrix algebra over $F$ if and only if $\mu \in \Nr(F(\sqrt[p]{\lambda})/F)$. 

Next assume $\ell = 0$. We will distinguish several cases in the description of $F^\alpha G$.
\begin{enumerate}
\item Assume $\lambda, \mu \in (F^*)^p$, i.e. the polynomials $X^p-\lambda$ and $X^p -\mu$ have roots in $F[X]$. Then $[\alpha]$ is the neutral element, as it corresponds to the trivial homomorphism in $\Hom(G/G', F^*/(F^*)^p)$. So $F^\alpha G \cong FG$. Now $FG$ has a trivial module $F$ and non-trivial modules all of which are isomorphic to $F(\zeta)$ as a field. There are $p+1$ possible kernels for a non-trivial module as this is the number of non-trivial subgroups of $G$. Moreover for a fixed kernel, say $\langle b \rangle$, there are $d$ non-equivalent representations of $FG$ on $F(\zeta)$ given by sending $a$ to $\zeta_i$ for $1 \leq i \leq d$. Hence
$$F^\alpha G \cong FG \cong F \oplus (p+1)dF(\zeta). $$
\item Assume $\lambda \notin (F^*)^p$ and $\mu \in (F^*)^p$. Then $[F(\sqrt[p]{\lambda}):F] = p$, as $X^p-\lambda$ is irreducible in $F[X]$ if it has no roots \cite[Theorem 427]{Redei}.  $F(\sqrt[p]{\lambda}, \zeta)$ is a field of dimension $\frac{p(p-1)}{d}$ over $F$. The following are pairwise non-equivalent projective irreducible $\alpha$-representations of $G$:
\begin{align*}
\eta &: G \rightarrow F(\sqrt[p]{\lambda}), \ x \mapsto \sqrt[p]{\lambda}, \ y \mapsto \sqrt[p]{\mu} \\
\eta_i &: G \rightarrow F(\sqrt[p]{\lambda}, \zeta), \ x \mapsto \sqrt[p]{\lambda}, \ y \mapsto \zeta_i\sqrt[p]{\mu}, \ \text{for} \ 1 \leq i \leq d. 
\end{align*}
Note that $[F(\sqrt[p]{\lambda}) : F] + d [F(\sqrt[p]{\lambda}, \zeta): F] = p + d\left(\frac{p(p-1)}{d}\right) = p^2 = |G|$. 
So $F^\alpha G \cong F(\sqrt[p]{\lambda}) \oplus dF(\sqrt[p]{\lambda}, \zeta)$.
\item Assume $\lambda, \mu \notin (F^*)^p$ and that $F(\sqrt[p]{\lambda}) \neq F(\sqrt[p]{\mu})$. Note, that if $F(\sqrt[p]{\lambda}) = F(\sqrt[p]{\mu})$ then after a base change in $F^\alpha G$ we can assume that we are in the previous case. The following is a projective irreducible $\alpha$-representation of $G$: 
$$\eta: G \rightarrow F(\sqrt[p]{\lambda}, \sqrt[p]{\mu}), \ x \mapsto \sqrt[p]{\lambda}, \ y \mapsto \sqrt[p]{\mu}.$$ 
Note that the dimension of $F(\sqrt[p]{\lambda}, \sqrt[p]{\mu})$ over $F$ is $p^2$. So $F^\alpha G \cong F(\sqrt[p]{\lambda}, \sqrt[p]{\mu})$.
\end{enumerate}

 \subsection{Non-abelian groups of order $p^3$}\label{sec:p^3} 
 In this section we prove Theorem 1. Throughout this section $p$ is an odd prime. 
 Abusing notation, we denote the two non-abelian groups of order $p^3$ by 
 $$G = \langle b \rangle  \rtimes \langle a \rangle, \ \ a^p = b^{p^2} = 1, \ b^p =: c, \ b^{a} = bc, \ c \in Z(G) $$
and 
 $$H = (\langle c \rangle \times \langle b \rangle ) \rtimes \langle a \rangle, \ \ a^p = b^p = c^p = 1, \ b^{a} = bc, \ c \in Z(H). $$
 In both groups the center is cyclic of order $p$ generated by $c$ and the quotient by the derived subgroup is elementary abelian of order $p^2$ generated by the images of $a$ and $b$. We denote by $\bar{a}$ and $\bar{b}$ the images of $a$ and $b$ in $G/G' = G/Z(G)$ and $H/H'=H/Z(H)$. We will study the twisted group algebras of $G$ and $H$ over $F$ starting with cohomology classes in the images of $\Ext(G/G', F^*)$ and $\Ext(H/H', F^*)$ under the inflation map respectively.

By \cite[Theorem 4.7.3]{KarpilovskyProjective} we have $M(G) = 1$ and $M(H) \cong C_p \times C_p$ where generators of $M(H)$ are given by 
\begin{align*}
 [\beta_a]&:\quad [u_c,u_a]=\zeta, \ [u_c,u_b]=1,\ [u_b,u_a]=u_z, \\
 [\beta_b]&:\quad [u_c,u_a]=1,\ [u_c,u_b]=\zeta,\ [u_b,u_a]=u_z.
\end{align*}

First let $[\alpha] \in H^2(G, F^*)$ be such that $\alpha$ is symmetric and given by $u_a^p = \lambda$, $u_b^p = \mu u_c$. Then $[\alpha]$ is inflated from $[\gamma] \in H^2(G/\langle c \rangle, F^*)$ which is given by $u_{\bar{a}}^p = \lambda$, $u_{\bar{b}}^p = \mu $, $[u_{\bar{a}}, u_{\bar{b}}] = 1$. By Theorem~\ref{th:NEWINFSTUFF} we have
$$F^\alpha G \cong \bigoplus_{\chi \in \Irr(\langle c \rangle, F)} F(\chi)^{\alpha_\chi} (G/ \langle c \rangle).$$
Now $\chi \in \Irr(\langle c \rangle, F)$ is determined by $\chi(c) = \zeta^\ell$ for some $0 \leq \ell \leq p-1$ and $|\Irr( \langle c \rangle, F)| = 1 + d$. For $\ell = 0$ we have $F(\chi)^{\alpha_\chi} (G/\langle c \rangle) = F^\gamma (G/\langle c \rangle)$ which is a direct sum of fields described in Section~\ref{sec:CpxCp}. If $\ell \neq 0$, then $[\alpha_\chi]$ can be described by relations $u_{\bar{a}}^p = \lambda$, $u_{\bar{b}}^p = \mu \zeta^\ell$, $[u_{\bar{a}}, u_{\bar{b}}] = \zeta^\ell$. The corresponding direct summand $F(\chi)^{\alpha_\chi}G$ is then the central simple $F(\zeta)$-algebra $(\lambda, \mu \zeta^\ell; F(\zeta); \zeta^\ell)_p$ as described in Section~\ref{sec:Algebras}. Note that $M(G)$ is trivial and hence by Section~\ref{sec:H2} any cohomology class in $H^2(G, F^*)$ contains a symmetric cocycle. Thus we have described all twisted group algebras of $G$ over $F$.
  
Next we describe all twisted group algebras of the Heisenberg group $H$. By the description of $M(H)$ above and Section~\ref{sec:H2} we can write any $[\alpha] \in H^2(H, F^*)$ as $[\alpha' \beta_a^i \beta_b^j]$ where $0 \leq i,j \leq p-1$ and $\alpha'$ is a symmetric cocycle explained by relations $u_a^p = \lambda$, $u_b^p = \mu$. There are essentially two cases which we will handle separately.

\textbf{Case 1:} $i = j = 0$, i.e. $[\alpha]$ contains a symmetric cocycle.

This case is similar to the above description of the twisted group algebra of $G$. More precisely, $[\alpha]$ is inflated from $[\gamma] \in H^2(H/\langle c \rangle, F^*)$ with $[\gamma]$ explained by relations $u_{\bar{a}}^p = \lambda$, $u_{\bar{b}}^p = \mu $, $[u_{\bar{a}}, u_{\bar{b}}] = 1$. Note that this is exactly the $[\gamma]$ from the description of twisted group algebras for $G$. Again by Theorem~\ref{th:NEWINFSTUFF} we have
$$F^\alpha H \cong \bigoplus_{\chi \in \Irr(\langle c \rangle, F)} F(\chi)^{\alpha_\chi} (H/ \langle c \rangle).$$
For $\chi(c) = 1$ we get the summand $F^\gamma (H/\langle c \rangle)$ which is isomorphic to $F^\gamma (G/\langle c \rangle)$. For $\chi(c) = \zeta^\ell$ we get the central simple $F(\zeta)$-algebra $(\lambda, \mu; F(\zeta); \zeta^\ell)_p$. Notice that the parameters of this algebra are not the same as in the description of twisted group algebras of $G$.  

\textbf{Case 2:} $(i, j) \neq (0,0)$.

Assume first $i \neq 0$. Let $k$ be an integer such that $ki + j \equiv 0 \bmod p$ and define in $F^\alpha H$ the element $v=u_{a^k}u_b$. Then a straightforward computation shows that $[u_c,v]=1$ and $v^p=\delta \in F^*$ for $\delta = \lambda^k \mu$. Next observe that $[v,u_a]=u_c$. So $F^{\alpha}H$ is generated by the elements $u_c,v,u_a$ with relations given by
$$u_c^p=1, v^p=\delta, u_a^p=\lambda, [v,u_a]=u_c, [v,u_c]=1, [u_c,u_a]=\zeta^i.$$
Now, notice that $e=\frac{1}{p}(1+u_c+\ldots +u_c^{p-1})$ is an idempotent, and the set 
$$e,e^{u_a}, \ldots, e^{u_a^{p-1}}$$ 
is a set of orthogonal idempotents such that their sum is $1$. Then, by Lemma~\ref{lem:PassmanMatrices} we have
$F^{\alpha}H \cong M_p(e(F^{\alpha}H)e)$.
Noticing that $eu_ce=e$ and $eu_ae=0$ we get that $e(F^{\alpha}H)e\cong \langle ve|(ve)^p=\delta \rangle$. Notice that $\langle ve \rangle$ is not only a subalgebra by dimension considerations. 
Hence 
$$F^{\alpha}H \cong \left\{ \begin{array}{cc} M_p(F(\sqrt[p]{\delta})),& \text{if} \ \ \sqrt[p]{\delta} \notin F \\  p M_p(F),& \text{otherwise} \end{array}\right.. $$
In case $i = 0$ we can make an analogues calculation interchanging $a$ and $b$ (and choosing $k=0$) so that  we then obtain
$$F^{\alpha}H \cong \left\{ \begin{array}{cc} M_p(F(\sqrt[p]{\lambda})),&  \text{if} \ \ \sqrt[p]{\lambda} \notin F \\  p M_p(F),& \text{otherwise} \end{array}\right. $$
We note that in this case all components are matrix algebras over fields.\\

\textit{Proof of Theorem 1:}
We show that $G \sim_F H$ if and only if $\zeta \notin F$ and $\zeta \in \Nr(F(\sqrt[p]{\lambda}, \zeta)/F(\zeta))$ for all $\lambda \in F^*$.
Note that $\zeta \in F$ if and only if there is an $[\alpha] \in H^2(H, F^*)$ such that $F^\alpha H$ has only non-commutative components by Proposition~\ref{prop:SalatimTheorem}. On the other hand $F^\beta G$ has a commutative component for every $[\beta] \in H^2(G, F^*)$. Thus $G \sim_F H$ implies $\zeta \notin F$.

So assume $\zeta \notin F$ from now on. Moreover, assume first that $\zeta \in \Nr(F(\sqrt[p]{\lambda}, \zeta)/F(\zeta))$ for all $\lambda \in F^*$. Let $\psi: H^2(H, F^*) \rightarrow H^2(G, F^*)$ be defined by sending an element $[\alpha]$ defined by the relations $u_a^p = \lambda$, $u_b^p = \mu$ to the element defined by $u_a^p = \lambda$, $u_b^p = \mu u_c$. Then by the above 
\[F^\alpha H \cong F^\gamma (C_p \times C_p) \oplus d (\lambda, \mu)_p\]
 and 
 \[F^{\psi(\alpha)} G \cong F^\gamma (C_p \times C_p) \oplus d (\lambda, \mu \zeta)_p\]
  for a certain $[\gamma] \in H^2(C_p \times C_p, F^*)$ and where the symbol algebras are defined over $F(\zeta)$. Now by Lemma~\ref{lem:SymbolAlgebras} we have $(\lambda, \mu)_p \cong (\lambda, \mu \zeta)_p$ if and only if $\zeta \in \Nr(F(\sqrt[p]{\lambda}, \zeta)/F(\zeta))$. So indeed $G \sim_F H$ via $\psi$.

Finally assume there is $\lambda \in F$ such that $\zeta \notin \Nr(F(\sqrt[p]{\lambda}, \zeta)/F(\zeta))$. Then also $\sqrt[p]{\lambda} \notin F(\zeta)$. 
Define $[\beta] \in H^2(G, F^*)$ via the relations $u_a^p = \lambda$, $u_b^p = 1$. So by the above 
$$F^\beta G \cong F(\sqrt[p]{\lambda}) \oplus d F(\sqrt[p]{\lambda}, \zeta) \oplus d (\lambda, \zeta)_p. $$
By Lemma~\ref{lem:SymbolAlgebras} the symbol algebra $(\lambda, \zeta)_p$ is not a matrix algebra over a field as $\zeta \notin \Nr(F(\sqrt[p]{\lambda}, \zeta)/F(\zeta))$. We claim that there is no $[\alpha] \in H^2(H,F^*)$ such that $F^\alpha H \cong F^\beta G$ and thus $G \not \sim_F H$. Indeed if $F^\alpha H \cong F^\beta G$, then $F^\alpha H$ has exactly $d+1$ summands which are fields and one of these fields is $F(\sqrt[p]{\lambda})$. Using the direct summand $F^\gamma(H/\langle c \rangle)$ described above, by Section~\ref{sec:CpxCp} this means that we can assume that $[\alpha]$ is given by relations $u_a^p = \lambda$ and $u_b^p = 1$. Note here that we can interchange $a$ and $b$ by an isomorphism of $H$. Then the non-commutative summands in  the Wedderburn decomposition of $F^\alpha H$ are isomorphic to $(\lambda,1)_p$ which is a matrix ring over a field by Lemma~\ref{lem:SymbolAlgebras}, so $F^\alpha H \not \cong F^\beta G$.
It now follows directly from Lemma~\ref{lem:Norm} that $G \not\sim_\Q H$. Hence Theorem 1 holds. \hfill \qed

\subsection{Answer to Question~\ref{que:QvsRest} and groups of order $p^4$}
We next study those groups of order $p^4$ which have isomorphic group algebras over $\mathbb{Q}$, and hence over all fields of characteristic different from $p$. This includes the proof of Theorem 2 which is a direct combination of Propositions~\ref{prop:Groups3^4Part1}, \ref{prop:Groups3^4Part2} and \ref{prop:Groupsp^4}. We consider all such groups with the exception of the direct product of a cyclic group of order $p$ and a non-abelian group of order $p^3$, as this case follows from Section~\ref{sec:p^3}. There are some differences between the cases $p=3$ and $p > 3$ and we will treat them separately, although the arguments are similar in both cases. We remark that the complex twisted group algebras for all groups of order $p^4$ were described in \cite{Higgs2006}.

First assume $p=3$ and consider the following groups of order $81$:
$$G_8=\langle a,b,c\ | \ a^9=b^3=c^3=1, [a,b]=c, [c,a] = 1, [c,b]=a^3\rangle,$$
$$G_9=\langle a,b,c\ | \ a^9=b^3=c^3=1, [a,b]=c, [c,a] = 1, [c,b]=a^{-3} \rangle,$$
$$G_{10}=\langle a,b,c \ | \ a^9=c^3=1, b^3=a^{-3}, [a,b] = c, [c,a] = 1, [c,b] =a^{-3} \rangle.$$
The numbering of the groups is coming from the Small Group Library. For all $i \in \{8,9,10\}$ we have $Z(G_i) = \langle a^3 \rangle$ and $G_i/Z(G_i) \cong H$ where $H$ denotes the Heisenberg group of order $27$. 
Furthermore, $G_i' = \langle a^3, c \rangle$ and $G_i/G_i' = \langle \bar{a} \rangle \times \langle \bar{b} \rangle \cong C_p \times C_p$. 
Moreover, by \cite[Theorem 3.1]{Higgs2006} we have $M(G_8) \cong M(G_{10}) \cong C_p$ and $M(G_9) \cong C_p \times C_p$. 

We first handle cohomology classes containing symmetric cocycles.
Let $[\alpha _i] \in \Ext(G_i/G_i', F^*)$. So by the properties of $G_i$ mentioned above $[\alpha _i]$ is determined by the following relations in the twisted group algebra for $i \in \{8,9 \}$
$$u_a^3=\lambda u_{a^3} \text{ and } u_b^3=\mu \text{ for } \lambda , \mu\in F^*$$
and for $i = 10$ by
$$u_a^3=\lambda u_{a^3} \text{ and } u_b^3=\mu u_{a^{-3}} \text{ for } \lambda , \mu\in F^*.$$

\begin{lemma}\label{lem:ExtCycle3^4}
	With the above notation
	$$F^{\alpha _8}G_8 \cong F^{\alpha _9}G_9 \cong F^{\alpha _{10}}G_{10}.$$
\end{lemma}

\begin{proof}
We will apply Theorem~\ref{th:NEWINFSTUFF} with the central subgroup $\langle a^3 \rangle$. So we have
$$F^{\alpha_i} G_i \cong \bigoplus_{\chi \in \Irr(\langle a^3 \rangle, F)} F(\chi)^{(\alpha_i)_\chi}H, $$
where 
$$H = (\langle \bar{c} \rangle \times \langle \bar{a} \rangle ) \rtimes \bar{b}, \ \ \bar{a}^3 = \bar{b}^3 = \bar{c}^3 = 1, \ \bar{a}^{\bar{b}} = \bar{a}\bar{c}, \ \bar{c} \in Z(H)$$
is the Heisenberg group. 
Recall from Section~\ref{sec:p^3} that an element in $H^2(H, F^*)$ is explained by relations $u_{\bar{a}}^3 = \lambda'$, $u_{\bar{b}}^3 = \mu'$, $[u_{\bar{c}}, u_{\bar{a}}] = \zeta^k$ and $[u_{\bar{c}}, u_{\bar{b}}] = \zeta^\ell$ for some $\lambda', \mu' \in F$ and where $k=\ell=0$ if $\zeta \notin F$.

First consider $\chi$ to be the principal character. Then $[(\alpha_i)_\chi]$ is, independently of $i$, explained by $u_{\bar{a}}^3 = \lambda$, $u_{\bar{b}}^3 = \mu$ and $u_{\bar{c}}$ is central. So the isomorphism type of this direct summand of $F^{\alpha_i}G_i$ does not depend on $i$.

Next let $\chi(a^3) = \zeta^\ell$ where $\zeta^\ell \neq 1$. So in particular, $F(\chi) = F(\zeta)$. Then $[(\alpha_i)_\chi]$ is explained by $u_{\bar{a}}^3 = \lambda \zeta^\ell$, $[u_{\bar{c}}, u_{\bar{a}}] = 1$ and 
$$u_{\bar{b}}^3 =  \left\{ \begin{array}{cc} \mu, & \text{if} \ \ i \in \{8,9\} \\  \mu \zeta^{-\ell} , & \text{if} \ \ i = 10  \end{array}\right. , \ \ \ \ \ [u_{\bar{c}}, u_{\bar{b}}] =  \left\{ \begin{array}{cc} \zeta^{-\ell}, &  \text{if} \ \ i \in \{9,10\} \\  \zeta^\ell , & \text{if} \ \ i = 8  \end{array}\right. .$$
By the last case considered in Section~\ref{sec:p^3} we then have, again independently of $i$, that 
$$F(\chi)^{(\alpha_i)_\chi}H \cong \left\{ \begin{array}{cc} M_p(F(\zeta, \sqrt[p]{\lambda \zeta^\ell})),&  \text{if} \ \ \sqrt[p]{\lambda \zeta^\ell} \notin F \\  p M_p(F(\zeta)),& \text{otherwise} \end{array}\right. .$$
So also in this case the isomorphism type of the direct summand corresponding to $\chi$ is independent of $i$ and the lemma follows.
\end{proof}

This allows us to determine over which fields these groups are in relation.

\begin{proposition}\label{prop:Groups3^4Part1}
Let $F$ be a field of characteristic different from $3$ and assume that $F$ contains no primitive $3$-rd too of unity. Then $G_8 \sim_F G_9 \sim_F G_{10}$. 
\end{proposition}

\begin{proof}
By the assumption on $F$ we have $\Hom(M(G_i), F^*) = 1$, so any cohomology class $[\alpha_i] \in H^2(G_i, F^*)$ contains a symmetric cocycle. Such a cohomology class can be explained by the relations  
$$u_a^3=\lambda u_{a^3} \text{ and } u_b^3=\mu \text{ for } \lambda , \mu\in F^*$$
for $i \in \{8,9\}$ and for $i = 10$ by
$$u_a^3=\lambda u_{a^3} \text{ and } u_b^3=\mu u_{a^{-3}} \text{ for } \lambda , \mu\in F^*.$$
Denote this cohomology class by $[\alpha_i^{\lambda, \mu}]$. Hence Lemma~\ref{lem:ExtCycle3^4} implies that for $i,j \in \{8,9,10\}$ the map 
\[\varphi_{i,j}: H^2(G_i, F^*) \rightarrow H^2(G_j, F^*),  \ \ [\alpha_i^{\lambda, \mu}] \mapsto [\alpha_j^{\lambda, \mu}] \]
is an isomorphism of groups such that $F^{\alpha_i^{\lambda, \mu}} G_i \cong F^{\alpha_j^{\lambda, \mu}} G_j$. We conclude that $G_8 \sim_F G_9 \sim_F G_{10}$.
\end{proof}

\begin{proposition}\label{prop:Groups3^4Part2}
Let $F$ be a field of characteristic different from $3$ containing a primitive $3$-rd root of unity. Then $G_8 \not\sim_F G_9$ and $G_8 \sim_F G_{10}$.
\end{proposition}

\begin{proof}
As $M(G_8) \cong C_3$ and $M(G_9) \cong C_3 \times C_3$, we obtain $\Hom(M(G_8), F^*) \cong C_3$ while $\Hom(M(G_{9}), F^*) \cong C_3 \times C_3$. Hence $G_8 \not\sim_F G_{9}$ by Proposition~\ref{prop:IsosFromSalatim}.
  
We need to show that $G_8 \sim_F G_{10}$. 
$M(G_8)$ and $M(G_{10})$ are both explained by the following cohomology classes:
\begin{align*}
[\beta_8]: [u_a,u_b] = u_c, [u_c,u_a] = \zeta, [u_c,u_b]=u_{a^3}, \\
[\beta_{10}]: [u_a,u_b] = u_c, [u_c,u_a] = \zeta, [u_c,u_b]=u_{a^{-3}}. 
\end{align*}
To show that $[\beta_8]$ is indeed a cohomology class notice that, $G_8=(\langle a \rangle\times \langle c \rangle) \rtimes \langle b\rangle$, and by our knowledge on the cohomology of an abelian group there is a cohomology class $[\tilde{\beta_8}]$ defined on $\langle a \rangle\times \langle c \rangle$ by 
$[u_c,u_a]=\zeta $. Hence, in order to show that $[\beta_8]$ is indeed a cohomology class it is sufficient, by Lemma~\ref{lemma:semicoprimeSchur}, to show that $[\tilde{\beta_8}]$ is invariant to the action of $b$ on $\langle a \rangle\times \langle c \rangle$. It is now a straightforward calculation to check that
$$[u_{b(c)},u_{b(a)}]=[u_{a^3c},u_{ac}]=\zeta$$
and hence $[\beta_8]$ is indeed a cohomology class of $G_8$. By similar arguments we can show that $[\beta_{10}]$ is a cohomology class of $G_{10}$.

Let a general element $[\alpha_8] \in H^2(G_8, F^*)$ be explained by the product of $[\beta_8]^r$ and a cohomology class containing a symmetric cocycle such that $u_a^3 = \lambda u_{a^3}$ and $u_b^3 = \mu$. So overall we can assume that $[\alpha_8]$ is explained by
\[ [u_a,u_b] = u_c, [u_c,u_a] = \zeta, [u_c,u_b]=u_{a^3}, u_a^3 = \lambda u_{a^3}, u_b^3 = \mu. \]
Let $\psi:H^2(G_8,F^*) \rightarrow H^2(G_{10},F^*)$ be a map defined by $\psi([\alpha_8]) = [\alpha_{10}]$ such that $[\alpha_{10}]$ is explained by the product of $[\beta_{10}]^r$ and a cohomology class containing a symmetric cocycle such that $u_a^3 = \zeta^r \lambda u_{a^3}$ and $u_b^3 = \mu^{-1} u_{a^{-3}}$. So $[\alpha_{10}]$ is explained by
 \[[u_a,u_b] = u_c, [u_c,u_a] = \zeta, [u_c,u_b]=u_{a^{-3}}, u_a^3 = \zeta^r \lambda u_{a^3}, u_b^3 = \mu^{-1} u_{a^{-3}}.  \]
 
We claim that $\psi$ is an isomorphism realizing the relation $G_8 \sim_F G_{10}$. It is easy to check that $\psi$ is an isomorphism of groups, so it remains to check that $F^{\alpha_8}G_8 \cong F^{\alpha_{10}}G_{10}$.
If $r \equiv 0 \bmod 3$, then we know $F^{\alpha_8} G_8 \cong F^{\alpha_{10}}G_{10}$ from Lemma~\ref{lem:ExtCycle3^4}. Note here that in Lemma~\ref{lem:ExtCycle3^4} we had $[\alpha_{10}]$ explained by $u_b^3=\mu u_{a^{-3}}$, but this difference does not influence the calculations. So assume $r \not\equiv 0 \bmod 3$.

Set $A = F^{\alpha_8}G_8$. The element $u_{a^3}$ is central in $A$ and satisfies $u_{a^3}^3 = 1$. Hence the central subalgebra $F\langle u_{a^3} \rangle$ is isomorphic to $F \oplus F \oplus F$ and contains three primitive central idempotents $f_0$, $f_1$ and $f_2$, such that 
\[A = Af_0 \oplus Af_1 \oplus Af_2 \]
and $f_iu_{a^3} = f_i\zeta^i$ for $0\leq i \leq 2$. Set $A_i = Af_i$. Then $A_0$ is a twisted group algebra of the Heisenberg group $G_8/\langle a^3 \rangle$ which is by Section~\ref{sec:p^3} isomorphic to $M_3(F(\sqrt[3]{\mu}))$, if $\sqrt[3]{\mu} \notin F$, and to $3M_3(F)$ otherwise.

For this paragraph let $i \in \{1,2 \}$. Recall that for an element $x$ in an $F$-algebra $X$ and a unit $u \in X$ we write $x^u = u^{-1}xu$.  Let $e_i = \frac{1}{3}(1+u_c+u_{c^2})f_i$. As $u_c^3 = 1$ this is an idempotent in $A_i$. Note that $u_af_i$ is a unit in $A_i$. As $[u_c,u_a]=\zeta^r$ the elements $e_i, e_i^{u_af_i}, e_i^{u_a^2f_i}$ are pairwise orthogonal idempotents which sum up to $f_i$. Hence $A_i \cong M_3(e_iA_ie_i)$ by Lemma~\ref{lem:PassmanMatrices}. Note that $e_iu_ae_i = 0$ and $e_iu_ce_i = e_i$. Let $v_i = u_au_b^{-ir}f_i$. We claim that $(e_if_i)^{v_i} = e_if_i$. Indeed
\begin{align*}
(e_i f_i)^{v_i} =& \frac{1}{3}(1+u_c+u_{c^2})^{u_au_b^{-ir}}f_i = \frac{1}{3}(1+\zeta^r u_c+ \zeta^{2r} u_{c^2})^{u_b^{-ir}}f_i \\
 =& \frac{1}{3}(1+\zeta^r u_{a^3}^{-ir} u_c+ \zeta^{2r} u_{a^3}^{-2ir} u_{c^2}) f_i \\
  =& \frac{1}{3}(1+\zeta^r \zeta^{-i^2r} u_c+ \zeta^{2r} \zeta^{-2i^2r} u_{c^2}) f_i =  e_if_i. 
\end{align*}
Hence $e_iv_ie_i = e_iv_i$ and $(e_iv_i)^3 = e_if_i \zeta^i \lambda \mu^{-ir}$. So
\[e_iA_ie_i = \left\{ \begin{array}{ll} F(\sqrt[3]{\zeta^i \lambda \mu^{-ir}}), & \text{if} \ \sqrt[3]{\zeta^i \lambda \mu^{-ir}} \notin F, \\ 3F, & \text{otherwise} \end{array} \right..  \]

In a similar way we can consider $F^{\alpha_{10}}G_{10}$. Here we can also define the $f_i$, $A_i$ and $e_i$ as for $G_8$. So we have $A_0 \cong M_3(F(\sqrt[3]{\mu^{-1}}))$, if $\sqrt[3]{\mu^{-1}} \notin F$, and $3M_3(F)$ otherwise. We define $v_i = u_au_b^{ir}f_i$, so that again $(e_if_i)^{v_i} = e_if_i$. Moreover,
\[(e_iv_ie_i)^3 = e_i \lambda \zeta^r u_{a^3} \mu^{-ir}u_{a^{-3}}^{ir} f_i e_i = e_i \zeta^{r+i-i^2r} \lambda \mu^{-ir} f_i e_i = e_i \zeta^{i} \lambda \mu^{-ir} f_i e_i. \]
So in both cases the three direct summands of the algebras are isomorphic.
\end{proof}

We will now consider the case $p >3$. The arguments here are similar and we will not give as many details as before, but now there are four groups involved. Again we use the numbering of the Small Groups Library for our groups and set
\begin{align*}
G_7 &=  \langle a,b,c,d \ | \ a^p = b^p = c^p = d^ p = 1, [a,b] = c, [c,a] = 1, [c,b] = d, d \in Z(G_7) \rangle, \\
G_8 &= \langle a,b,c \ | \ a^{p^2} = b^p = c^p = 1, [a,b]=c, [c,a] = a^p, [c,b] = 1 \rangle 
\end{align*}
and
\begin{align*}
G_9  &= \langle a,b,c \ | \ a^{p^2} = b^p = c^p = 1, [a,b]=c, [c,a] = 1, [c,b] = a^p \rangle  \\
G_{10} &= \langle a,b,c \ | \ a^{p^2} = b^p = c^p = 1, [a,b]=c, [c,a] = 1, [c,b] = a^{pm} \rangle 
\end{align*}
where $m$ is a quadratic nonresidue modulo $p$.

Then $Z(G_7) = \langle d \rangle$ and $Z(G_i) = \langle a^p \rangle$ for $i \in \{8,9,10 \}$. Moreover $G_i/Z(G_i) \cong H$ where $H$ denotes the Heisenberg group for all $i \in \{7,8,9,10 \}$. Furthermore $G_i' = \langle Z(G_i), c \rangle$ and $G_i/G_i' = \langle \bar{a} \rangle \times \langle \bar{b} \rangle \cong C_p \times C_p$, for all $i$. Finally $M(G_7) \cong C_p \times C_p$ while $M(G_i) \cong C_p$ for $i \in \{8,9,10 \}$
by \cite[Theorem 3.1]{Higgs2006}.

Again we will first consider cohomology classes containing symmetric cocycles. By the above we can assume that $[\alpha_i] \in \Ext(G_i, F^*)$ is determined by relations
$$u_a^p = \left\{ \begin{array}{cc}  \lambda, & \text{if}  \ \ i=7   \\ \mu^{-1} u_{a^p}, & \text{if}  \ \ i=8 \\ \lambda u_{a^p}, & \text{if} \ \ i \in \{9,10 \} \end{array} \right. \ \ \text{and} \ \ 
 u_b^p =   \left\{ \begin{array}{cc}  \mu, & \text{if}  \ \ i \in \{7,9,10 \}   \\ \lambda, & \text{if}  \ \ i=8 \end{array} \right.$$
for $\lambda, \mu \in F^*$.

\begin{lemma}\label{lem:ExtCyclep^4}
With the above notation
	$$F^{\alpha_7}G_7 \cong F^{\alpha_8}G_8$$
	and
	$$ F^{\alpha_{9}}G_9 \cong F^{\alpha_{10}}G_{10}.$$	
\end{lemma}
\begin{proof}
As in the situation for $p=3$ we will apply Theorem~\ref{th:NEWINFSTUFF} with the central subgroup $Z(G_i)$. Set $Z(G_i) = \langle z \rangle$. So we have
$$F^{\alpha_i} G_i \cong \bigoplus_{\chi \in \Irr(\langle z \rangle, F)} F(\chi)^{(\alpha_i)_\chi}H, $$
where 
$$H = (\langle \bar{c} \rangle \times \langle \bar{a} \rangle ) \rtimes \bar{b}, \ \ \bar{a}^3 = \bar{b}^3 = \bar{c}^3 = 1, \ \bar{a}^{\bar{b}} = \bar{a}\bar{c}, \ \bar{c} \in Z(H)$$
is the Heisenberg group. 

Let $\chi$ be the principal character. Then $[(\alpha_i)_\chi]$ is explained by
$$u_{\bar{a}}^p = \left\{ \begin{array}{cc}  \lambda, & \text{if}  \ \ i \in \{7,9,10\}   \\ \mu^{-1}, & \text{if}  \ \ i=8 \end{array} \right. \ \ \text{and} \ \ 
 u_{\bar{b}}^p =   \left\{ \begin{array}{cc}  \mu, & \text{if}  \ \ i \in \{7,9,10 \}   \\ \lambda, & \text{if}  \ \ i=8 \end{array} \right.$$ 
and $u_{\bar{c}}$ is central in all cases. From our previous considerations for the twisted group algebras of the Heisenberg group and the elementary abelian group of order $p^2$ we see that the isomorphism type of this direct summand does not depend on $i$. This follows from Lemma~\ref{lem:SymbolAlgebras}, as this gives $(\lambda, \mu)_p \cong (\mu^{-1},\lambda)_p$.

Next let $\chi(z) = \zeta^\ell$ where $\zeta^\ell \neq 1$. Then $[(\alpha_i)_\chi]$ is explained by 
$$u_{\bar{a}}^p = \left\{ \begin{array}{cc}  \lambda, & \text{if}  \ \ i=7   \\ \mu^{-1} \zeta^\ell, & \text{if}  \ \ i=8 \\ \lambda \zeta^\ell, & \text{if} \ \ i \in \{9,10 \} \end{array} \right. \ \ \text{,} \ \ 
 u_{\bar{b}}^p =   \left\{ \begin{array}{cc}  \mu, & \text{if}  \ \ i \in \{7,9,10 \}   \\ \lambda, & \text{if}  \ \ i=8 \end{array} \right.$$
and
$$[u_{\bar{c}}, u_{\bar{a}}] =  \left\{ \begin{array}{cc} 1, &  \text{if} \ \ i \in \{ 7,9,10\} \\  \zeta^\ell , & \text{if} \ \ i = 8  \end{array}\right. \ \ , \ \ [u_{\bar{c}}, u_{\bar{b}}] =  \left\{ \begin{array}{ccc} \zeta^\ell, &  \text{if} \ \ i \in \{7,9\} \\  1 , & \text{if} \ \ i = 8 \\ \zeta^{\ell m} , & \text{if} \ \ i = 10 \end{array}\right. .$$

From our considerations for the Heisenberg group we see that $F(\chi)^{(\alpha_7)_\chi} H \cong F(\chi)^{(\alpha_8)_\chi} H$ and $F(\chi)^{(\alpha_9)_\chi} H \cong F(\chi)^{(\alpha_{10})_\chi} H$ which implies the lemma.
\end{proof}

We obtain a result similar to the one for the case $p=3$.

\begin{proposition}\label{prop:Groupsp^4}
Let $p$ be a prime bigger than $3$ and let $F$ be a field of characteristic different from $p$. Then
\begin{enumerate}
\item $G_7 \sim_F G_8$ if and only if $F$ does not contain a primitive $p$-th root of unity. 
\item $G_9 \sim_F G_{10}$.
\end{enumerate}
\end{proposition}
\begin{proof}
First assume $F$ contains no primitive $p$-th root of unity. Then\\ $\Hom(M(G_i),F^*)$ is the trivial group for all $i$ and it follows from Lemma~\ref{lem:ExtCyclep^4} that $G_7\sim_F G_8$ and $G_9 \sim_F G_{10}$. Now assume $F$ does contain a primitive $p$-th root of unity $\zeta$. As $M(G_7) \cong C_p \times C_p$ while $M(G_8) \cong C_p$ we conclude from Proposition~\ref{prop:IsosFromSalatim} that $G_7 \not\sim_F G_8$. So it remains to show that $G_9 \sim_F G_{10}$.

We have $M(G_9) \cong M(G_{10}) \cong C_p$ and the Schur multipliers are generated by a cohomology class explained by
\begin{align*}
[\beta_9]:& \ [u_a,u_b] = u_c, [u_c,u_a]=\zeta, [u_c, u_b] = u_{a^p}, \\
[\beta_{10}]:& \ [u_a,u_b] = u_c, [u_c,u_a]=\zeta, [u_c, u_b] = u_{a^{pm}} \\
\end{align*}
for $G_9$ and $G_{10}$ respectively. It can be checked as in the proof of Proposition~\ref{prop:Groups3^4Part2} that $[\beta_9]$ and $[\beta_{10}]$ indeed define cohomology classes of the corresponding group.

Let a general cohomology class $[\alpha_9]$ in $H^2(G, F^*)$ be explained by the product of $[\beta_9]^r$ with the cohomology class containing a symmetric cocycle defined by $u_a^p = \lambda u_{a^p}$ and $u_b^p = \mu$. Let $\psi:H^2(G_9, F^*) \rightarrow H^2(G_{10},F^*)$ be a map and $\psi([\alpha_9]) = [\alpha_{10}]$ where $[\alpha_{10}]$ is explained by the product of $[\beta_{10}]^r$ and a cohomology class containing a symmetric cocycle defined by the relations $u_a^p = \lambda u_{a^p}$ and $u_b^p = \mu^m$. We claim that $\psi$ is a group isomorphism realizing the relation $G_9 \sim_F G_{10}$. 
If $r \equiv 0 \bmod p$, then $F^{\alpha_9}G_9 \cong F^{\alpha_{10}}G_{10}$ by Lemma~\ref{lem:ExtCyclep^4}. 

So let $r \not\equiv 0 \bmod p$ and set $A = F^{\alpha_9}G_9$. As $u_{a^p}$ is central in $A$ satisfying $u_{a^p}^p = 1$ the subalgebra $F\langle u_{a^p} \rangle$ is isomorphic to the direct sum of $p$ copies of $F$ such that each direct summand contains an idempotent $f_i$ which is central in $A$ and such that $u_{a^p} f_i = \zeta^i f_i$ for $0 \leq i \leq p-1$. Hence $A = \oplus_{i=0}^{p-1} Af_i$. Set $A_i = Af_i$. Then $A_0$ is a twisted group algebra of the Heisenberg group $G_9/\langle a^p \rangle$ which is by Section~\ref{sec:p^3} isomorphic with $M_p(F(\sqrt[p]{\mu}))$, if $\sqrt[p]{\mu} \notin F$, and $pM_p(F)$ otherwise.

For this paragraph let $1 \leq i \leq p-1$. Set $e_i = \frac{1}{p}(1+u_c+\ldots+u_c^{p-1})f_i$ which is an idempotent in $A_i$. Moreover $e_i,e_i^{u_af_i},\ldots,e_i^{u_a^{p-1}f_i}$ is a set of pairwise orthogonal idempotents which sum up to $f_i$. Hence by Lemma~\ref{lem:PassmanMatrices} we know $A_i \cong M_p(e_iA_ie_i)$. Let $j$ be an integer such that $ij \equiv 1 \bmod p$ and set $v_i = u_au_b^{-rj}f_i$. Then $v_i$ is centralizing $e_if_i$ as
\begin{align*}
(e_if_i)^{v_i} &= \frac{1}{p}(1+u_c+\ldots+u_c^{p-1})^{u_au_b^{-rj}}f_i =  \frac{1}{p}(1+\zeta^r u_c+\ldots + \zeta^{(p-1)r} u_c^{p-1})^{u_b^{-rj}}f_i \\
 &= \frac{1}{p}(1+\zeta^r u_{a^p}^{-rj} u_c+\ldots + \zeta^{(p-1)r} u_{a^p}^{-(p-1)rj} u_c^{p-1})f_i \\
 &= \frac{1}{p}(1+\zeta^{r-rij} u_c+\ldots + \zeta^{(p-1)r-(p-1)rij} u_c^{p-1})f_i = e_if_i.
\end{align*}    
Moreover $v_i^p = \zeta^i \lambda \mu^{-rj}$, so $A_i \cong M_p(F(\sqrt[p]{\zeta^i \lambda \mu^{-rj}}))$, if $\sqrt[p]{\zeta^i \lambda \mu^{-rj}} \notin F$, and $pM_p(F)$ otherwise.

To determine the isomorphism type of $F^{\alpha_{10}}G_{10}$ we can define the $f_i$, $A_i$ and $e_i$ in the same way. Then $A_0 \cong M_p(F(\sqrt[p]{\mu^m}))$, if $\sqrt[p]{\mu^m} \notin F$, and $pM_p(F)$ otherwise. Note that $F(\sqrt[p]{\mu}) = F(\sqrt[p]{\mu^m})$. To define the $v_i$ choose $j'$ such that $j'im \equiv 1 \bmod p$ and set $v_iu_au_b^{-rj'}$. Then a similar computation as before shows that $v_i$ centralizes $e_if_i$ and we have $v_i^p = \zeta^i \lambda \mu^{-mrj'} f_i$. Note that $\mu^{-mrj'} \equiv \mu^{-rj} \bmod (F^*)^p$. Hence $F(\sqrt[p]{\zeta^i \lambda \mu^{-rj}}) = F(\sqrt[p]{\zeta^i \lambda \mu^{-mrj'}})$ and we conclude that indeed $F^{\alpha_8}G_9 \cong F^{\alpha_{10}}G_{10}$. 
\end{proof}

\section{Dade's examples and proof of Theorem 3}\label{Dade}
In 1971 Dade answered a question of Brauer by describing a series of finite groups $G$ and $H$ such that $FG \cong FH$ for all fields $F$ \cite{Dade}. In \cite[Section 5]{MargolisSchnabel2} we studied the (TGRIP) for the examples of Dade and could show for a subclass of these examples that there exists a finite field $F$ such that $G \not\sim_F H$. Our goal here will be to show that in fact $G \not\sim_\mathbb{Q} H$ for any of Dade's examples, i.e. to prove Theorem 3. 

Let us first describe the groups of Dade. Let $p$ and $q$ be primes such that $q \equiv 1 \mod p^2$ and let $w$ be an integer such that $w \not \equiv 1 \mod q^2$, but $w^p \equiv 1 \mod q^2$. Let $Q_1$ and $Q_2$ be the following two non-abelian groups of order $q^3$.
\begin{align*}
Q_1 &= (\langle \tau_1 \rangle \times \langle \sigma_1 \rangle) \rtimes \langle \rho_1 \rangle, \\
Q_2 &= \langle \sigma_2 \rangle \rtimes \langle \rho_2 \rangle, \\
\tau_1^q &= \sigma_1^q = \rho_1^q = \sigma_2^{q^2} = \rho_2^q = 1, \ \sigma_2^q =: \tau_2, \\
\tau_1^{\rho_1} &= \tau_1, \ \sigma_1^{\rho_1} = \tau_1\sigma_1, \ \sigma_2^{\rho_2} = \tau_2\sigma_2.
\end{align*}
So $Q_1$ and $Q_2$ are just the two non-abelian groups of order $q^3$ such that $Q_1$ has exponent $q$. These are exactly the groups for which we studied the (TGRIP) in Section~\ref{sec:p^3} and the difference between the groups we encountered in Theorem 1 is going to be crucial here.

Let $\langle \pi_1 \rangle \cong C_{p^2}$, $\langle \pi_2 \rangle \cong C_p$ and for $i,j \in \{1,2 \}$ let
$$\rho_i^{\pi_j} = \rho_i, \ \sigma_i^{\pi_j} = \sigma_j^w, \ \tau_i^{\pi_j} = \tau_i^w.$$
Define two groups by
\begin{align*}
G &= (Q_1 \rtimes \langle \pi_1 \rangle) \times (Q_2 \rtimes \langle \pi_2 \rangle), \\
H &= (Q_1 \rtimes \langle \pi_2 \rangle) \times (Q_2 \rtimes \langle \pi_1 \rangle).
\end{align*}
These are the groups constructed by Dade as a counterexample to Brauer's question and we will fix them throughout this section.

Notice that $G=G_1\times G_2$ and $H=H_1\times H_2$ for
$$G_1=Q_1 \rtimes \langle \pi_1 \rangle,\quad G_2=Q_2 \rtimes \langle \pi_2 \rangle,\quad
 H_1=Q_1 \rtimes \langle \pi_2 \rangle, \quad H_2=Q_2 \rtimes \langle \pi_1 \rangle.$$

The Schur multipliers of $G$ and $H$ are described in \cite[Proposition 5.3]{MargolisSchnabel2}, but it will turn out that they are in fact not needed to obtain $G \not\sim_\mathbb{Q} H$. Note that, using the bar notation for taking elements modulo the commutator subgroup, for $i \in \{1,2 \}$ we have 
\[ G_i/G_i' = \langle \overline{\rho_i} \rangle \times \langle \overline{\pi_i} \rangle\]
and 
\[ H_i/H_i' = \langle \overline{\rho_i} \rangle \times \langle \overline{\pi_j} \rangle\]
where $j \in  \{1,2 \}$ such that $i \neq j$. Let $\lambda \in F^*$. Abusing notation we define a cohomology class $[\alpha_\lambda]$ in $H^2(G_1,F^*)$ and $H^2(H_1,F^*)$ and a cohomology class $[\beta_\lambda]$ in $H^2(G_2, F^*)$ and $H^2(H_2,F^*)$. These cohomology classes are in the image of $\Ext(G_i/G_i', F^*)$ and $\Ext(H_i/H_i', F^*)$ under the inflation map, respectively, and are described by the relations
\[[\alpha_\lambda]: u_{\rho_1}^q = \lambda \ \ \text{and} \ \ [\beta_\lambda]: u_{\rho_2}^q = \lambda. \]
We will also write simply $[\alpha_\lambda]$ and $[\beta_\lambda]$ when we restrict these cohomology classes to subgroups of $G$ and $H$ which contain $\rho_1$ or $\rho_2$ respectively.

\textbf{Notation:} For a positive integer $n$ let $\zeta_n$ denote a primitive complex $n$-th root of unity. Let $F$ be a field of characteristic $0$ such that $F \cap \Q(\zeta_{q^2}) = \Q$. We denote $\xi = \sum_{i=0}^{p-1} \zeta_q^{wi}$. Moreover, we denote $\delta_{p} \in \Gal(F(\zeta_q)/F(\xi))$, explained by $\delta_p(\zeta_q) = \zeta_q^w$, and $\delta_{pq} \in \Gal(F(\zeta_{q^2})/F(\xi))$, explained by $\delta_{pq}(\zeta_{pq}) = \zeta_{q^2}^{(q+1)w}$.

We first determine the isomorphism types of the twisted group algebras of those groups which form Dades's examples. 

\begin{lemma}\label{lem:DadePartsIsoTypesQ}
 Let $\lambda \in \mathbb{Q}$. Assume $\lambda \not\in (\Q^*)^q$ and set $z = \sqrt[q]{\lambda}$. With the cohomology classes defined above the following isomorphisms hold {\small
\begin{align*}
&\Q  H_1 \cong \Q G_2 \cong \Q \oplus \Q(\zeta_p) \oplus \Q(\zeta_q) \oplus \Q(\zeta_p, \zeta_q) \oplus M_p(\Q(\xi)) \oplus \frac{q-1}{p}M_p(\Q(\zeta_q)) \oplus M_{pq}(\Q(\xi)), \\
&\Q^{\alpha_\lambda} H_1 \cong \Q(z) \oplus \Q(z,\zeta_p) \oplus M_p(\Q(z,\xi)) \oplus M_{pq}(\Q(\xi)), \\
&\Q^{\beta_\lambda} G_2 \cong \Q(z) \oplus \Q(z,\zeta_p) \oplus M_p(\Q(z,\xi)) \oplus (\lambda, \Q(\zeta_{q^2})/\Q(\xi), \delta_{pq}), \\
&\Q  G_1 \cong \Q H_2 \cong \Q H_1 \oplus \Q(\zeta_{p^2}) \oplus \Q(\zeta_q, \zeta_{p^2}) \oplus (\zeta_p, \Q(\zeta_q, \zeta_p)/\Q(\xi, \zeta_p), \delta_p) \oplus   \frac{q-1}{p}M_p(\Q(\zeta_q, \zeta_p)) \\
& \ \ \ \ \ \ \ \ \  \ \ \ \ \ \ \ \ \ \ \ \ \ \ \  \  \oplus M_{q}((\zeta_p, \Q(\zeta_q, \zeta_p)/\Q(\xi, \zeta_p), \delta_p)), \\
&\Q^{\alpha_\lambda} G_1 \cong \Q^{\alpha_\lambda} H_1 \oplus \Q(z,\zeta_{p^2}) \oplus (\zeta_p, \Q(\zeta_q, z, \zeta_p)/\Q(\xi, z, \zeta_p), \delta_p) \oplus M_q((\zeta_p, \Q(\zeta_q, \zeta_p)/\Q(\xi, \zeta_p), \delta_p)), \\
&\Q^{\beta_\lambda} H_2 \cong \Q^{\beta_\lambda} G_2 \oplus \Q(z,\zeta_{p^2}) \oplus (\zeta_p, \Q(\zeta_q, z, \zeta_p)/\Q(\xi, z, \zeta_p), \delta_p) \oplus (\zeta_p \lambda, \Q(\zeta_{q^2}, \zeta_p)/\Q(\xi, \zeta_p),\delta_{pq}). 
\end{align*}}
\end{lemma}

\begin{proof}
We will not prove the isomorphisms of the non-twisted group algebras, as these can be derived by standard methods, e.g. this can be achieved using the theory of Shoda pairs as presented in \cite[Chapter 3]{JespersdelRio}.

Note that by assumption $z \notin \Q$. Set $e = \frac{1}{q}(1+\tau_1+\ldots+\tau_1^{q-1})$ which is a central idempotent in $\Q^{\alpha_\lambda} H_1$. So we have
\begin{equation}\label{eq:QalphalambdaH1}
\Q^{\alpha_\lambda} H_1 \cong e\Q^{\alpha_\lambda} H_1e \oplus (1-e)\Q^{\alpha_\lambda} H_1(1-e). 
\end{equation}
Now 
\[e\Q^{\alpha_\lambda} H_1e \cong \Q^{\alpha_\lambda} ( \langle \sigma_1 \rangle \rtimes \langle \pi_2 \rangle \times \langle \rho_1 \rangle) \cong \Q(\langle \sigma_1 \rangle \rtimes \langle \pi_2 \rangle) \otimes \Q^{\alpha_\lambda}\langle \rho_1 \rangle.	 \]
While $\Q^{\alpha_\lambda} \langle \rho_1 \rangle \cong \Q(z)$, the isomorphism type of the untwisted group algebra is given by
\begin{equation}\label{eq:SigmaPi2}
\Q(\langle \sigma_1 \rangle \rtimes \langle \pi_2 \rangle) \cong \Q \oplus \Q(\zeta_p) \oplus M_p(\Q(\xi)). 
\end{equation}
Hence, we obtain 
\[e\Q^{\alpha_\lambda} H_1e \cong \Q(z) \oplus \Q(z,\zeta_p) \oplus M_p(\Q(z,\xi)).\]

So by \eqref{eq:QalphalambdaH1} it remains to consider the algebra $(1-e)\Q^{\alpha_\lambda} H_1(1-e)$ and to show that it is isomorphic to $M_{pq}(\Q(\xi))$.
Let $f = \frac{1}{q}(1+\sigma_1+\ldots+\sigma_1^{q-1})$, so $(1-e)f(1-e)$ is an idempotent. The idempotents 
\[(1-e)f(1-e),(1-e)f^{\rho_1}(1-e),\ldots,(1-e)f^{\rho^{q-1}}(1-e)\]
are then pairwise orthogonal and sum up to $(1-e)$. Hence by Lemma~\ref{lem:PassmanMatrices} we have 
\[(1-e)\Q^{\alpha_\lambda} H_1(1-e) \cong M_q((1-e)f\Q^{\alpha_\lambda} H_1f(1-e)).\]
As $f\rho_1f = 0$, we have 
\[(1-e)f\Q^{\alpha_\lambda} H_1f(1-e) \cong (1-e) \Q(\langle \tau_1 \rangle \rtimes \langle \pi_2 \rangle) (1-e).\]

By \eqref{eq:SigmaPi2} this algebra is isomorphic to $M_p(\Q(\xi))$, where we use that\\ $\langle \sigma_1 \rangle \rtimes \langle \pi_2 \rangle \cong \langle \tau_1 \rangle \rtimes \langle \pi_2 \rangle$. Overall we obtain the claimed isomorphism type of $\Q^{\alpha_\lambda} H_1$.

Now set $e = \frac{1}{q}(1+\tau_2+\ldots+\tau_2^{q-1})$. Similarly as before we have
\begin{align*}
&\Q^{\beta_\lambda} G_2 \cong e\Q^{\beta_\lambda} G_2e \oplus (1-e)\Q^{\beta_\lambda} G_2(1-e)\\ 
&\cong \Q(z) \oplus \Q(z,\zeta_p) \oplus M_p(z,\xi) \oplus (1-e)\Q^{\beta_\lambda} G_2(1-e) 
\end{align*}
and we need to prove that
\[(1-e)\Q^{\beta_\lambda} G_2(1-e) \cong (\lambda, \Q(\zeta_{q^2})/\Q(\xi), \delta_{pq}). \]
Let $a \in (\lambda, \Q(\zeta_{q^2})/\Q(\xi), \delta_{pq})$ such that $a^n = \lambda$ and $ \mu a = a \delta_{pq}(\mu)$ for all $\mu \in \Q(\zeta_{q^2})$. The algebra $(1-e)\Q^{\beta_\lambda} G_2(1-e)$ has a faithful representation in $(\lambda, \Q(\zeta_{q^2})/\Q(\xi), \delta_{pq})$ by sending $\sigma_2$ to $\zeta_{q^2}$ and $\rho_2 \pi_2$ to $a$. As both these algebras have the same dimension, we conclude that they are in fact isomorphic.

The isomorphism types of $\Q^{\alpha_\lambda} G_1$ and $\Q^{\beta_\lambda} H_2 $ follow by similar arguments and using 
\begin{equation*}
 \Q(\langle \sigma_1 \rangle \rtimes \langle \pi_1 \rangle)\cong \Q(\langle \sigma_1 \rangle \rtimes \langle \pi_2 \rangle) \oplus \Q(\zeta_{p^2}) \oplus (\zeta_p, \Q(\zeta_q, \zeta_p)/\Q(\xi, \zeta_p), \delta_p)
\end{equation*}
 where we used \eqref{eq:SigmaPi2} before. 
\end{proof}

We will also use another field than $\Q$ to simplify our argument later.
\begin{corollary}\label{cor:DadePartsIsoTypesF}
Let $F = \Q(\zeta_q, \zeta_p)$, $\lambda \in \Q$ and assume $\lambda \not\in (F^*)^q$. Set $z = \sqrt[q]{\lambda}$. With the cohomology classes described above the following isomorphisms hold
\begin{align*}
&F H_1 \cong F G_2 \cong pqF \oplus  \frac{q(q-1)}{p} M_p(F) \oplus \frac{q-1}{p}M_{pq}(F), \\
&F^{\alpha_\lambda} H_1 \cong pF(z) \oplus \frac{q-1}{p} M_p(F(z)) \oplus \frac{q-1}{p} M_{pq}(F), \\
&F^{\beta_\lambda} G_2 \cong pF(z) \oplus \frac{q-1}{p} M_p(F(z)) \oplus (q-1)p (\lambda, \zeta_q; F; \zeta_q)_q, \\
&FG_1 \cong FH_2 \cong pqF \oplus q(p-1)F(\zeta_{p^2}) \oplus q(q-1)M_p(F) \oplus (q-1)M_{pq}(F), \\
&F^{\alpha_\lambda} G_1 \cong pF(z) \oplus (p-1)F(z,\zeta_{p^2}) \oplus (q-1)M_p(F(z)) \oplus (q-1)M_{pq}(F), \\
&F^{\beta_\lambda} H_2 \cong pF(z) \oplus (p-1)F(z,\zeta_{p^2}) \oplus (q-1)M_p(F(z)) \oplus p^2(q-1)( \lambda, \zeta_q;F;\zeta_q)_q.
\end{align*}
\end{corollary}
\begin{proof}

The claimed isomorphism types follow from Lemma~\ref{lem:DadePartsIsoTypesQ} by tensoring with $F$. We will just comment on how the cyclic algebras appearing in Lemma~\ref{lem:DadePartsIsoTypesQ} behave under this change of coefficients. As $[F:\Q(\xi, \zeta_p)] = p$, by Proposition~\ref{prop:CyclicAlgebra} we know that $F \otimes (\zeta_p, \Q(\zeta_q, \zeta_p)/\Q(\xi, \zeta_p), \delta_p)$ is Brauer equivalent to $(\zeta_p, F/F, \text{id})$ which is isomorphic to $F$. So 
\[F \otimes (\zeta_p, \Q(\zeta_q, \zeta_p)/\Q(\xi, \zeta_p), \delta_p) \cong \frac{(q-1)(p-1)}{p}M_p(F).\]

Moreover, also by Proposition~\ref{prop:CyclicAlgebra}, $F \otimes (\lambda, \Q(\zeta_{q^2})/\Q(\xi), \delta_{pq})$ is Brauer equivalent to $(\lambda, \Q(\zeta_{q^2})/F, \delta_q)$ where $\delta_q$ is a generating element of $\Gal(\Q(\zeta_{q^2})/F))$. The latter algebra is just the symbol algebra $(\lambda, \zeta_q;F;\zeta_q)_q$. Finally in the same manner 
\[F \otimes (\lambda, \Q(\zeta_{q^2})/\Q(\xi), \delta_{pq}) \equiv (\zeta_p \lambda, \zeta_q;F;\zeta_q)_q.\]
Now $\Nr_{\Q(\zeta_{q^2})/F}(\zeta_p) = \zeta_p^q$, so $\zeta_p \in \Nr(\Q(\zeta_{q^2})/F)$. We conclude from Lemma~\ref{lem:SymbolAlgebras} that 
\[(\zeta_p \lambda, \zeta_q;F;\zeta_q)_q \cong (\lambda, \zeta_q;F;\zeta_q)_q.\]
\end{proof}

In order to prove Theorem 3 we will use the following lemma.

\begin{lemma}\label{lem:NormDade}
There exists a $\lambda \in F^*$ such that $(\lambda, \zeta_q;F;\zeta_q)_q$ is not Brauer equivalent to a field. 
\end{lemma}
\begin{proof}
By Lemma~\ref{lem:SymbolAlgebras} we have $(\lambda, \zeta_q;F;\zeta_q)_q \equiv 1$ if and only if $\zeta_q \in \Nr(F(z)/F)$. By Lemma~\ref{lem:Norm} however, there exists $\lambda \in \Q^*$ such that $\zeta_q \notin \Nr(F(z)/F)$. Thus, the proof is complete. 
\end{proof}

\textit{Proof of Theorem 3:} Let $\lambda \in F$ such that $(\lambda, \zeta_q;F;\zeta_q)_q$ is not Brauer equivalent to a field, i.e. $(\lambda, \zeta_q;F;\zeta_q)_q$ is not isomorphic to a matrix ring over a field. Such a $\lambda$ exists by Lemma~\ref{lem:NormDade}. We claim that there exists no $[\gamma] \in H^2(G, \Q^*)$ such that $\Q^\gamma G \cong \Q^{\alpha_\lambda} H$. If such a $[\gamma]$ exists, then also $F^\gamma G \cong F^{\alpha_\lambda} H$, so it will suffice to prove that $F^\gamma G \not \cong F^{\alpha_\lambda} H$ for any $[\gamma]$. So assume for a contradiction that $F^\gamma G \cong F^{\alpha_\lambda} H$ for some $[\gamma] \in H^2(G, F^*)$. 

By Corollary~\ref{cor:DadePartsIsoTypesF} we have that $F^{\alpha_\lambda} H \cong F^{\alpha_\lambda} H_1 \otimes FH_2$ is isomorphic to
\begin{align*}
&\left( pF(z) \oplus \frac{q-1}{p} M_p(F(z)) \oplus \frac{q-1}{p} M_{pq}(F) \right) \\
 & \ \ \  \ \ \ \ \ \ \ \otimes \left( pqF \oplus q(p-1)F(\zeta_{p^2}) \oplus q(q-1)M_p(F) \oplus (q-1)M_{pq}(F)\right) . 
\end{align*}

We first note that the Wedderburn decomposition of $F^{\alpha_\lambda}H$ consists solemnly of components isomorphic to matrix rings over fields. Moreover it contains the commutative component $F(z)$. It follows from Proposition~\ref{prop:SalatimTheorem} that $[\gamma]$ is in the image of the inflation map of $\Ext(G/G',F^*)$ in $H^2(G,F^*)$. Recall that
\[G/G' = \langle \overline{\rho_1} \rangle \times \langle \overline{\rho_2} \rangle \times \langle \overline{\pi_1} \rangle \times \langle \overline{\pi_2} \rangle.  \]
So $\Ext(G/G', F^*)$ has exponent $pq$. Assume that the order of $[\gamma]$ is divisible by $p$. Then there exists a $\mu \in F^*$ such that $\mu \notin (F^*)^p$ and we could assume that in $F^\gamma G$ one has $u_{\rho_1}^p = \mu$ or $u_{\rho_2}^p = \mu$. Then $F^\gamma G$ would contain a Wedderburn component which is a field containing $\sqrt[p]{\mu}$. This is not the case. Hence $[\gamma]$ must have order $q$. 

So we can assume that $[\gamma] = [\alpha_{\lambda_1} \beta_{\lambda_2}]$ for some $\lambda_1, \lambda_2 \in F$. By Corollary~\ref{cor:DadePartsIsoTypesF} we then know that $F^\gamma G$ contains a component isomorphic to $F(\sqrt[q]{\lambda_1}, \sqrt[q]{\lambda_2})$. Looking at the components of $F^{\alpha_\lambda} H$ we see that $F(\sqrt[q]{\lambda_1}, \sqrt[q]{\lambda_2}) \cong F(z) $ which implies that we can assume $\{\lambda_1, \lambda_2 \} \subseteq \{1, \lambda \}$. Note that $F^{\alpha_\lambda} H$ contains exactly $p^2q$ components of type $F(z)$. If $\lambda_1 = \lambda_2 = 1$, then by Corollary~\ref{cor:DadePartsIsoTypesF} the algebra $F^\gamma G$ contains no such component, while when $\lambda_1 = \lambda_2 = \lambda$ it contains exactly $p^2$ such components. So we conclude $\{\lambda_1, \lambda_2 \} = \{1,\lambda \}$.  

Assume that $\lambda_2 = \lambda$. Then by Corollary~\ref{cor:DadePartsIsoTypesF} the algebra $F^\gamma G$ contains a component $(\lambda, \zeta_q;F;\zeta_q)_q$. By our assumption on $\lambda$ this component is not isomorphic to a matrix ring over a field, contradicting the fact that all components of $F^{\alpha_\lambda}H$ are isomorphic to matrix rings over fields. 

Hence we can assume $\lambda_2 =1$ and $\lambda_1 = \lambda$. 
By Corollary~\ref{cor:DadePartsIsoTypesF} we then know that $F^\gamma G =  F^{\alpha_\lambda} G \cong F^{\alpha_\lambda} G_1 \otimes FG_2$ contains a component 
\[F(z,\zeta_{p^2}) \otimes M_{pq}(F) \cong M_{pq}(F(z,\zeta_{p^2})).\] 
But analyzing the Wedderburn decomposition of $F^{\alpha_\lambda} H$ we obtain that it does not contain such a component, namely the only components which are isomorphic to $(pq \times pq)$-matrix rings over a field are $M_{pq}(F)$ and $M_{pq}(F(z))$. Hence $F^{\alpha_\lambda} H \not\cong F^{\alpha_\lambda} G$ and overall $F^{\alpha_\lambda} H$ is not isomorphic to $F^\gamma G$ for any $[\gamma] \in H^2(G,F^*)$. \hfill \qed

\begin{remark}\label{rem:RoggenkampAndHertweck}
To our knowledge additionally to Dade's example there are two more examples in the literature of pairs of groups which have isomorphic group algebras over any field. These are on one hand the famous counterexamples to the integral isomorphism problem by Hertweck \cite{HertweckIso} and on the other hand some much simpler and less known examples of Roggenkamp \cite[Chapter VIII]{RoggenkampTaylor}. It is an open question whether these examples satisfy the relation of the (TGRIP) over any field.

On the other hand in view of the fact that $H^2(G, F^*) = 1$ for $G$ a $p$-group and $F$ a perfect field of characteristic $p$, it is tempting to look for $p$-groups $G$ and $H$ such that $G \sim_F H$ holds for any field $F$. This includes the search for $p$-groups $G$ and $H$ such that $FG \cong FH$ holds for any field $F$ of characteristic $p$. The question if such groups exist is known as the Modular Isomorphism Problem for which first counterexamples were found recently \cite{GarciaLucasMargolisDelRioCounterexample}. These examples however do not satisfy $\mathbb{Q}G \cong \mathbb{Q}H$, as mentioned in \cite{MargolisMIPSurvey}, so they are also no candidates to solve Problem~\ref{TwistedBrauer}.
\end{remark}

\textbf{Thanks:} We thank Adam Chapman, Yuval Ginosar and Danny Neftin for many fruitful conversations. We also express our gratitude to the Technion and the Vrije Universiteit Brussel, and especially Eric Jespers, which helped arrange visits of the authors to each others institutions.

\bibliographystyle{amsalpha}
\bibliography{Rational}

\end{document}